\documentclass[11pt,reqno]{amsart}
\usepackage[toc,page]{appendix}
\usepackage{svg}
\usepackage{graphicx} 
\usepackage{url}
\usepackage{amsmath} 
\usepackage{amssymb}
\usepackage{amsthm}
\usepackage{bm, bbm}
\usepackage{booktabs}
\usepackage{polynom}
\usepackage[outline]{contour}
\usepackage{xcolor}
\usepackage{pgfplots}
\usepackage{mathtools}
\usepackage{stmaryrd}
\usepackage{caption}
\pgfplotsset{compat=1.18,} 
\usepackage{xfrac}  
\usepackage{enumitem} 
\usepackage{float}
\usepackage{comment}
\usepackage{mathrsfs}
\usepackage[automark]{scrlayer-scrpage}
\usepackage[ruled,vlined, linesnumbered]{algorithm2e}
\usepackage{changepage}
\usepackage[normalem]{ulem}

\usepackage{multirow}
\usepackage{booktabs}

\usepackage{pifont}

\usepackage{adjustbox}

\DeclareMathAlphabet{\mathpzc}{OT1}{pzc}{m}{it}

\theoremstyle{plain}
\newtheorem{thm}{Theorem}[section] 

\newtheorem{prop}[thm]{Proposition}
\newtheorem{rem}[thm]{Remark}

\newtheorem{defn}[thm]{Definition}

\definecolor{unia-purple}{RGB}{173, 0, 124}
\definecolor{light-gray}{rgb}{0.8889,0.8889,0.8889}
\definecolor{cpl3}{rgb}{1.0,0.0,0.0}
\definecolor{cpl1}{rgb}{0.8889,0.4356,0.2781}
\definecolor{cpl2}{rgb}{0.0,0.6056,0.9787}
\definecolor{cpl4}{rgb}{0.2422,0.6433,0.3044}
\definecolor{cpl6}{rgb}{0.2, 0.2, 0.2}

\usepackage{scalerel,stackengine}
\usepackage[maxnames=5]{biblatex} 
\renewbibmacro{in:}{}
\DeclareFieldFormat[article,incollection,inproceedings,phdthesis,thesis,misc,unpublished]{title}{#1}  
\AtEveryBibitem{%
	\clearfield{number}
	{}}

\AtEveryBibitem{%
	\ifentrytype{article}{%
		\clearfield{url}%
	}{}
}

\usepackage{hyperref}
\hypersetup{
	colorlinks=true,
	linkcolor=blue,
	filecolor=blue,      
	urlcolor=blue,
	citecolor=blue
}

\newcommand\N{\mathbb N}
\newcommand\C{\mathbb C}

\newcommand\R{\mathbb R}

\DeclareMathOperator{\Drm}{D}
\DeclareMathOperator{\grad}{grad}

\DeclareMathOperator{\logit}{logit}

\DeclareMathOperator{\trace}{trace}
\DeclareMathOperator{\id}{id}

\DeclareMathOperator*{\argmin}{arg\,min}

\DeclareMathOperator{\diag}{diag}
\DeclareMathOperator{\dom}{dom}
\DeclareMathOperator{\rank}{rank}
\DeclareMathOperator{\res}{res}
\DeclareMathOperator{\tol}{tol}

\def\restr{r}
\def\prol{p}
\def\onebb{\mathbbm{1}}

\def\drm{\text{d}}
\def\dxi{\,\text{d}\xi}
\def\ds{\,\text{d}s}
\def\dt{\,\text{d}t}

\newcommand\calA{\mathcal A}
\newcommand\calB{\mathcal B}
\newcommand\calE{\mathcal E}
\newcommand\calC{\mathcal C}

\newcommand\calK{\mathcal K}

\newcommand\calS{\mathcal S}
\newcommand\calT{\mathcal T}
\newcommand\calU{\mathcal U}

\newcommand{\M}{\mathcal{M}}

\newcommand{\Tan}{\mathrm{T}}
\newcommand{\Ret}{\mathcal R}
\newcommand{\invRet}{\mathcal L}
\newcommand{\St}{\mathrm{St}}

\newcommand{\GP}{{\rm GP}}
\newcommand{\KS}{{\rm KS}}
\newcommand{\CC}{{\rm CC}}

\newcommand{\bsm}{\left(\begin{smallmatrix}}
	\newcommand{\esm}{\end{smallmatrix}\right)}

\makeatletter
\def\author@andify{%
	\nxandlist {\unskip ,\penalty-1 \space\ignorespaces}%
	{\unskip {} \@@and~}%
	{\unskip \penalty-2 \space \@@and~}%
}
\renewcommand{\andify}{%
	\nxandlist{\unskip, }{\unskip{} \@@and~}{\unskip \space \@@and~}}
\makeatother

\definecolor{myBlue}{rgb}{0.0, 0.0, 0.8}
\definecolor{myOrange}{RGB}{225,92,22} 
\definecolor{myGreen2}{RGB}{114,175,30} 
\definecolor{myPurple}{rgb}{0.494,0.184,0.556}%

\newif\ifshowcomments
\showcommentstrue

\ifshowcomments
\newcommand{\Ycomment}[1]{{\color{magenta}{\bf Y:} #1}}
\newcommand{\Jcomment}[1]{{\color{myOrange}{\bf J:} #1}}
\newcommand{\Scomment}[1]{{\color{myGreen2}{\bf S:} #1}}
\newcommand{\Tcomment}[1]{{\color{myBlue}{\bf T:} #1}}
\newcommand{\Vcomment}[1]{{\color{myPurple}{\bf V:} #1}}
\newcommand{\todo}[1]{{\color{red}#1}}

\else
\newcommand{\Ycomment}[1]{}
\newcommand{\Jcomment}[1]{}
\newcommand{\Scomment}[1]{}
\newcommand{\Tcomment}[1]{}
\newcommand{\Vcomment}[1]{}
\newcommand{\todo}[1]{}
\fi

\usepackage{tikz}
\usepackage{tikz-cd}
\usepackage{tikz-3dplot}
\usetikzlibrary{
	external, 
	arrows.meta, decorations.pathmorphing, backgrounds, fadings, calc}
\tikzset{external/system call={pdflatex -shell-escape -halt-on-error -interaction=nonstopmode -jobname "\image" "\texsource" \space \relax}}

\title[]{Riemannian Multilevel Optimization with Application to 
	Constrained Energy Minimization Problems}

\author[Y. Elshiaty]{Yara Elshiaty}
\author[S. Petra]{Stefania Petra}
\author[J. P\"uschel]{Jonas P\"uschel}
\author[T. Stykel]{Tatjana Stykel}
\author[F.-J. Vanmaele]{Ferdinand-Joseph Vanmaele}
\address[Y. Elshiaty]{Institute for Mathematics, Heidelberg University, Im Neuenheimer Feld 205, 69120, Heidelberg, Germany, \tt{elshiaty@math.uni-heidelberg.de}}
\address[S. Petra]{Institute for Mathematics, IWR \& MIISM, Heidelberg University, Im Neuenheimer Feld 205, 69120, Heidelberg, Germany, \tt{petra@math.uni-heidelberg.de}}
\address[J. P\"uschel]{Institute of Mathematics, University of 
	Augsburg, Universit\"atsstra{\ss}e~12a, 86159 Augsburg, Germany, \tt{jonas.pueschel@uni-a.de}}
\address[T. Stykel]{Institute of Mathematics \& Centre for 
	Advanced Analytics and Predictive Sciences (CAAPS), University of 
	Augsburg, Universit\"atsstra{\ss}e~12a, 86159 Augsburg, Germany, \tt{tatjana.stykel@uni-a.de}}
\address[F.-J. Vanmaele]{Institute for Mathematics, IWR \& MIISM, Heidelberg University, Im Neuenheimer Feld 205, 69120, Heidelberg, Germany,
	\tt{ferdinand.vanmaele@medma.uni-heidelberg.de}}

\thanks{Y.~Elshiaty and S.~Petra acknowledge support by the German Research Foundation under Germany’s Excellence Strategy EXC-2181/1 - 390900948 (the Heidelberg STRUCTURES Excellence Cluster). 
	The work of J.~P\"uschel and T.~Stykel is part of a project that has received funding from the German Research Foundation – Project number 564828373. 
}

\begin{document}
	
	\begin{abstract}
		Multilevel optimization methods are highly effective for discretized energy minimization problems, but their Euclidean formulation does not directly apply to manifold constraints. We introduce a Riemannian extension of multilevel optimization based on a  coarse model that is first-order coherent with the fine-level objective and yields descent directions under mild retraction-convexity assumptions. The framework includes metric-compatible vector transfer operators for passing first-order information between levels, covering both intrinsic and extrinsic constructions. We formulate two-level and multilevel algorithms and prove global convergence using a Riemannian Zoutendijk-type argument. Applications to Kohn--Sham density functional theory, Gross--Pitaevskii ground-state computation, and binary continuous cuts demonstrate the method on Stiefel, ellipsoid and Bernoulli manifolds. The experiments show significant reductions in computational time compared with single-level Riemannian optimization.
	\end{abstract}
	
	\maketitle
	
	
	
	
	\section{Introduction}
	\label{sec:introduction}
	
	We consider constrained minimization problems of the form
	\begin{equation}\label{eq:min}
		\min_{x\in\M} f(x),
	\end{equation}
	where $f:V\to\R$ is a smooth objective function, $V$ is a~finite-dimensional Euclidean space, and $\M\subset V$ is a~feasible set that admits a~smooth Riemannian manifold structure. Such problems arise in many applications in which an energy or loss functional is minimized under structural constraints imposed by the underlying model. Examples include quantum systems, electronic structure calculations, statistical inference, computer vision, image processing, and variational models for partial differential equations. After choosing a~finite-dimensional representation, such as finite elements, plane waves, or graph-based discretizations,  one obtains optimization problems with structured feasible sets. In multilevel settings, this structure is inherited across levels, yielding a hierarchy of constrained optimization problems. Here, the term level is understood broadly and is not restricted to geometric grids. 
	
	We aim to develop efficient optimization methods for solving~\eqref{eq:min} based on first-order iterative schemes, with the goal of improving computational efficiency. 
	The key idea is to exploit a hierarchy of coarse representations of~\eqref{eq:min} while carefully transferring information between discretization levels.
	
	The main difficulty in multilevel optimization under constraints is that coarse corrections typically violate feasibility, and maintaining a~descent direction within the constraint set is nontrivial. These challenges can be addressed by introducing maps between points on the coarse and fine manifolds, together with consistent vector transfer operators for tangent vectors. Using these multilevel structures, we construct efficient Riemannian optimization algorithms that reduce computational cost by improving the convergence properties of the fine-level problem.
	
	\subsection{Related work}
	\label{ssec:RelatedWork}
	
	In the context of energy minimization and related nonlinear eigenvalue problems, multilevel and multigrid techniques have attracted considerable attention. A~widely used strategy is cascadic multigrid, e.g., \cite{AntLT17,BorH08,WuWB17}, closely related to coarse-to-fine or multiscale 
	acceleration strategies in imaging \cite{RouxLeclercHild2014}. The problem is 
	solved on a hierarchy of increasingly fine discretization levels, with each 
	coarse-level solution prolongated to initialize the next finer optimization 
	problem. While simple and broadly applicable, this strategy does not provide a 
	coarse correction of the current fine-level objective. Since the objective 
	itself changes with the discretization level, decrease on one level does not 
	directly imply decrease on the next finer level.
	
	In contrast to purely cascadic strategies, more tightly coupled multilevel
	methods use coarse discretizations to compute corrections for the current
	fine-level iterate. Examples include two-level discretization approaches 
	for nonlinear eigenvalue and constrained energy minimization problems
	\mbox{\cite{CanCHM18,CanDMSV16,GriH25,HenMP14}}, Newton-based multigrid
	methods~\cite{XuXXF21}, and adaptive multilevel strategies based on a~posteriori error estimators~\cite{ChenDGHZ14,wang2023posteriori}, mesh
	redistribution~\cite{XieXYZ23}, or local energy reduction~\cite{HeidW25}.
	
	A different class of multilevel methods is represented by the \emph{full approximation scheme}, originally introduced by Brandt \cite{Brandt1977}. It is based on solving a~nonlinear equation on the coarse level, where the coarse gradient is corrected by restricted fine-level gradient information. The fine iterate is then updated using a~prolongated coarse-grid correction. Alternatively, the \emph{optimization formulation}, due to Nash~\cite{Nash:2000}
	and commonly referred to as MGOPT in the multigrid literature, replaces this nonlinear coarse equation by the minimization of a~coarse surrogate functional, obtained by augmenting the coarse objective with a linear correction term. These two approaches are equivalent in the sense that the nonlinear coarse equation corresponds to the first-order optimality condition of the surrogate minimization problem.
	
	The optimization formulation has several advantages that are important for this
	paper. First, the coarse problem need not be solved exactly: a sufficiently
	decreasing approximate coarse solution already yields a valid fine-level
	correction; see Proposition~\ref{prop:descent_dir}. Second, the resulting
	correction can be combined naturally with a line search, which is the basis of
	the convergence analysis for line-search MGOPT methods in~\cite{WenGoldfarb2010}.
	Third, and central to the present paper, the optimization formulation extends beyond unconstrained Euclidean problems. Constraints in a Euclidean setting can be incorporated by the Bregman approach of~\cite{ElshiatyP2026}, while manifold constraints require the correction term to be expressed using the geometry of the feasible set rather than by directly comparing Euclidean residuals. Riemannian extensions of this idea were developed in~\cite{MuePZ23,SutV2021}.
	
	The approach of~\cite{SutV2021} adapts 
	the multilevel optimization formulation of~\cite{Nash:2000} to embedded manifolds and formulates the coarse correction extrinsically. 
	The work~\cite{MuePZ23} develops a~Riemannian multilevel method
	including geometric Galerkin-type coarse models and the distinction between point and vector interlevel transfers, and illustrates it on Bernoulli manifolds with prolongation-based transfer constructions.
	The present paper builds on these developments and provides a unified Riemannian multilevel optimization framework that covers both intrinsic and extrinsic constructions, several metric-compatible transfer operators, and a~global convergence analysis.
	
	\subsection{Contribution and organization}
	\label{ssec:contribution}
	
	We develop a Riemannian multilevel optimization framework for constrained minimization problems.  While inspired by the optimization perspective on multigrid methods, the proposed framework is not limited to grid-based discretizations. Instead, the hierarchy may be induced by finite element spaces, plane-wave bases, graph-based discretizations, or algebraic coarse representations. The main contributions of this work are as follows.
	
	\begin{enumerate}
		\item \textit{Riemannian coarse model.} 
		We construct a~Riemannian coarse model for constrained optimization on Riemannian manifolds and show that it is first-order coherent with the fine-level objective; see Section~\ref{sec:coarse_model}.
		\item \textit{Vector transfer operators.} We develop a systematic family of vector transfer operators, including restriction-based, prolongation-based, and projection-based constructions, covering both intrinsic and extrinsic settings; see Section~\ref{sec:vector_transport_choice}.
		\item \textit{Metric-independent formulations.} We identify conditions on the vector transfer operators under which the coarse model is independent of the choice of Riemannian metric; see Proposition~\ref{prop:metric_independence}.
		\item \textit{Convergence analysis.} We prove global convergence of the resulting two-level and multilevel methods by means of  a Riemannian Zoutendijk-type argument; see Section~\ref{ssec:convergence}.
		\item \textit{Numerical validation.} 
		We demonstrate the flexibility of the framework on three constrained energy minimization problems with different manifold geometries and discretization types: Kohn--Sham density functional theory on the Stiefel manifold with plane waves, Gross--Pitaevskii ground-state computation on an ellipsoid with finite elements, and binary continuous cuts on the Bernoulli manifold with finite differences; see Section~\ref{sec:applications}.
	\end{enumerate}
	
	The paper is organized as follows.
	Section~\ref{sec:fundam} recalls retractions and lifting maps on Riemannian manifolds and defines retraction convexity.
	In Section~\ref{sec:coarse_model}, 
	we introduce the Riemannian coarse model and study the properties needed for the multilevel framework. Section~\ref{sec:vector_transport_choice} discusses the construction of different vector transfer operators that ensure a~compatible 
	exchange of first-order information between levels.
	In Section~\ref{sec:geometric-multilevel}, we present a~two-level Riemannian optimization method and extend it to a~multilevel framework.
	Section~\ref{sec:applications} provides comprehensive numerical results for three energy minimization problems from different application areas.
	Section~\ref{sec:conclusion} concludes the paper with a~summary of the main results and a~discussion of possible directions for future research.
	
	
	\section{Preliminaries}
	\label{sec:fundam}
	We briefly review the fundamental concepts of Riemannian geometry. For a~comprehensive treatment, the interested reader is referred to~\cite{AbsMS08,Boumal23}.
	A~\emph{Riemannian manifold} is a~tuple~$(\M, \langle\cdot,\cdot\rangle)$, where $\M$ is a smooth manifold, and $\langle\cdot,\cdot\rangle$ is a \emph{Riemannian metric}. The metric assigns to each point $x \in \M$ an inner product $\langle\cdot, \cdot\rangle_x$ on the tangent space $\Tan_x\M$ and this assignment depends smoothly on $x$. Consequently, for all $x \in \M$, $\Tan_x\M$ is a Hilbert space with respect to the inner product $\langle\cdot, \cdot\rangle_x$, which induces the norm $\|\cdot\|_x$.
	
	\subsection{Retractions and lifting maps}
	
	Retractions provide a~way to move on the manifold~$\M$ along tangent vectors. They are first-order approximations of the exponential map. For~\mbox{$x\in\M$}, a~\emph{retraction} is a smooth map $\Ret_x \colon \Tan_x\M \to \M$ that satisfies $\Ret_x(0) = x$ and $\Drm \Ret_x(0) = \id_{\Tan_x\M}$ at the origin $0 \in \Tan_x\M$.
	The differential of the retraction $\Ret_x$ at $v\in\Tan_x\M$,  
	\begin{equation*}
		\Drm \Ret_x(v) \colon \Tan_x\M \to \Tan_{\Ret_x(v)}\M, 
	\end{equation*}
	provides a vector transport, called the \emph{differentiated retraction vector transport}. In general, this mapping is not an~isometry; however, it is still an~isomorphism for all~$v$ in a~neighborhood of $0 \in \Tan_x\M$, whose existence is guaranteed by~\cite[Cor.~10.27]{Boumal23}. 
	
	A~\emph{lifting map} $\invRet_x\colon \calU_x \to \Tan_x\M$ is defined as a~local inverse of the retraction satisfying 
	\begin{equation}\label{eq:RL}
		\Ret_x(\invRet_x(y)) = y\qquad \text{ for all } y\in \calU_x,
	\end{equation}
	where $\calU_x\subset\M$ is an~open neighborhood of $x$. The differential of the lifting map $\invRet_x$ at $y\in \calU_x$,
	\[
	\Drm \invRet_x(y) \colon \Tan_y\M \to \Tan_{x}\M, 
	\]
	serves also as a~vector transport. Due to the smoothness of $\Ret_x$ and $\Drm \Ret_x(0) = \id_{\Tan_x\M}$, there exists a~neighborhood $\calU_x$ of $x$ such that $\Drm \Ret_x(\invRet_x(y))$ is invertible for all $y \in \calU_x$. In this case, applying the chain rule to \eqref{eq:RL}, we calculate the derivative of $\invRet_x$ as
	\begin{equation*} 
		\Drm \invRet_x(y) = \Drm \Ret_x(\invRet_x(y)) ^{-1}.
	\end{equation*}
	The exponential and logarithm maps are examples of retraction and lifting, respectively.
	
	\subsection{Retraction convexity}
	The concept of geodesic convexity \cite[Sect.~6]{Lee18} can be naturally extended to general retractions. 
	For a given retraction $\Ret$, a~set $\calC \subset \M$ is called \emph{retraction-convex} (or \emph{$\Ret$-convex}), if for any $x,y \in \calC$, $v = \invRet_x(y)$ exists and $\Ret_x(tv) \in \calC$ for all $t \in [0,1]$. The existence of $\Ret$-convex neighborhoods
	around every point of $\M$ was established in~\cite[Thm.~3.1]{SegK24}. On an~$\Ret$-convex set $\calC\subset \M$, a~function $f : \calC  \to \R$ is called \emph{$\Ret$-convex}, if for all $x,y \in \calC$ and $v = \invRet_x(y)$, the function $t \mapsto f(\Ret_x(tv))$ is convex on $[0,1]$ in the usual Euclidean sense. As in the Euclidean setting, $\Ret$-convexity admits a first-order characterization in the form of a supporting inequality.
	
	\begin{prop}[$\Ret$-supporting inequality]\label{prop:R_supp_ineq}
		Let $\Ret$ be a retraction on a~Riemannian mani\-fold~$\M$ and let $\invRet$ be the corresponding lifting map. Further, let $\calC \subset \M$ be a~$\Ret$-convex set and let $f : \calC \to  \R$ be continuously differentiable and $\Ret$-convex. Then for all $x,y \in \calC$, we have 
		\begin{equation}\label{eq:supp_ineq}
			f(y) \ge f(x) + \langle \grad f(x), \invRet_x(y) \rangle_x.
		\end{equation}
	\end{prop}
	\begin{proof}
		Let $x,y \in \calC$ and $v = \invRet_x(y)$. 
		For $t \in [0,1]$, define $\gamma(t) = \Ret_x(tv)$ and $f_\gamma(t) = f(\gamma(t))$. By definition of  $\Ret$-convexity, $f_\gamma : [0,1] \to \R$ is convex implying that $f_\gamma(1) \ge f_\gamma(0) + \dot{f}_\gamma(0)$. Using chain rule, we calculate
		\[
		\dot{f}_\gamma(0) = \frac{\drm}{\dt}  f(\gamma(t))\bigg|_{t=0}
		= \langle \grad f(x), \dot \gamma (0) \rangle_x 
		= \langle \grad f(x), v \rangle_x.
		\]
		Plugging this into the above inequality yields the result.
	\end{proof}
	
	This is a natural result, since the supporting inequality \eqref{eq:supp_ineq} depends only on first-order quantities, which are preserved by retractions through the first-order coherence condition $\Drm \Ret_x(0)=\id_{\Tan_x\M}$. In contrast, second-order quantities, such as Hessian-based characterizations and curvature, are in general not transferable to the $\Ret$-convexity framework.
	
	
	\section{Riemannian coarse model}\label{sec:coarse_model}
	
	The central idea of multilevel optimization is to accelerate iterations on the fine-level problem by occasionally substituting the fine objective $f_h$ with a~cheaper surrogate defined on a~coarser discretization. For the surrogate to be effective, two requirements must be satisfied. First, it has to retain the dominant nonlinear features of $f_h$ near the current iterate, so that its minimization yields a~direction along which $f_h$ is also expected to decrease. Second, it has to agree with $f_h$ to first order to ensure that the multilevel iteration does not introduce spurious critical points and inherits the convergence properties of the fine-level method. The classical construction satisfying both requirements in the Euclidean unconstrained setting is the Nash coarse model~\cite{Nash:2000}. Let~$f_h$ and~$f_H$ be the fine and coarse objective functions defined on $\R^{n_h}$ and $\R^{n_H}$, respectively, with $n_h>n_H$. Given current iterates $x \in \R^{n_h}$ and $y \in \R^{n_H}$, and a~linear restriction operator $R \in \R^{n_H \times n_h}$, one minimizes the coarse model
	\begin{equation}\label{eq:nash_eucl}
		q(z) = f_H(z) - \big\langle \nabla f_H(y) - R\nabla f_h(x),\, z - y \big\rangle.
	\end{equation}
	The linear correction term is chosen so that $\nabla q(y) = R\nabla f_h(x)$, thereby ensuring that the gradients of the coarse and fine problems coincide at the current iterates, up to the restriction operator~$R$.
	
	Two features of~the coarse model~\eqref{eq:nash_eucl} prevent its direct use on a~Riemannian manifold. The displacement $z - y$ is not intrinsically defined when both $z$ and $y$ belong to the coarse manifold, and the Euclidean inner product is no longer the natural pairing between the gradient and the displacement. Both objects, however, admit natural Riemannian counterparts: the displacement $z - y$ is replaced by the lifting  map $\invRet_y^H(z)$, which in tangent coordinates describes the motion from~$y$ to~$z$ along a~retraction. Similarly, the Euclidean inner product is replaced by the Riemannian metric $\langle \cdot, \cdot \rangle_y$ at the base point, while the linear restriction~$R$ is replaced by a~Riemannian vector restriction operator $R_x^y$ that transfers tangent vectors at~$x$ in the fine manifold to tangent vectors at~$y$ in the coarse manifold. With these modifications, the Nash construction extends almost verbatim and yields the Riemannian coarse model that forms the focus of this section. In what follows, we formulate the Riemannian coarse model precisely and show that it preserves first-order consistency (Proposition~\ref{prop:coherence}), generates descent directions for the fine objective (Proposition~\ref{prop:descent_dir}), and 
	for a fixed vector prolongation with metric-adjoint restriction, is independent of the choice of Riemannian metrics on the coarse and fine manifolds (Proposition~\ref{prop:metric_independence}).
	
	Let $V_H$ and $V_h$ denote the ambient Euclidean spaces associated with the coarse and fine spatial discretizations, respectively, with mesh widths~$H>h$. Furthermore, let $\M_H\subset V_H$ and $\M_h\subset V_h$ be the coarse and fine Riemannian manifolds equipped with the metrics $\langle\cdot,\cdot\rangle_y$ and~$\langle\cdot,\cdot\rangle_x$, respectively. The corresponding retractions are denoted by $\Ret^H$ and~$\Ret^h$. Consider the continuously differentiable coarse and fine objectives
	\[
	f_H \colon \M_H \to  \R, \qquad\qquad f_h\colon\M_h \to \R,
	\]
	with the Riemannian gradients $\grad f_H(y)\in \Tan_{y} \M_H$ and $\grad f_h(x)\in\Tan_{x} \M_h$. At each
	iteration of a~multilevel optimization scheme, we consider iterates
	$y \in \M_H$ and $x \in \M_h$ together with the injective \emph{vector
		prolongation} $P_{y}^{x} \colon \Tan_{y} \M_H \to \Tan_{x} \M_h$,
	and the surjective \emph{vector restriction} $R_{x}^{y} \colon \Tan_{x} \M_h \to \Tan_{y} \M_H$, which satisfy the \emph{geometric Galerkin condition} $R_{x}^{y} = (P_{y}^{x})^\ast$, i.e.,\
	\begin{equation}\label{eq:riem_galerkin}
		\langle P_{y}^{x} u,\, v\rangle_{x} \;=\; \langle u,\, R_{x}^{y} v\rangle_{y}
		\qquad \text{for all}\;u \in \Tan_{y} \M_H,\;v\in \Tan_{x} \M_h.
	\end{equation}
	These vector transfer operators enable the consistent transport of search
	directions and gradient information across levels, and the Galerkin
	condition ensures coherence of the multilevel optimization process.
	We further assume that the lifting map $\invRet_{y}^H=(\Ret_{y}^H)^{-1}$ is
	well-defined on a~$\Ret^H$-convex set $\calC_{y}\subset \M_H$. The
	\emph{Riemannian coarse model} then reads
	\begin{equation} \label{eq:riem_coarse_model}
		\min\limits_{z \in \calC_{y}} \; q(z) = f_H(z) - \langle w, \invRet^H_{y}(z)\rangle_{y}
		\enskip \text{with }\; w = \grad f_H(y) - R_{x}^{y}(\grad f_h(x)),
	\end{equation}
	where the correction term $\langle w, \invRet^H_{y}(z)\rangle_{y}$ is the
	Riemannian counterpart of the linear correction in~\eqref{eq:nash_eucl}
	and enforces first-order consistency at $z = y$
	(verified in Proposition~\ref{prop:coherence} below).
	
	Given an~approximate solution $\tilde{z}\in\calC_y$ of~\eqref{eq:riem_coarse_model}
	with $q(\tilde{z})<q(y)$, the resulting fine search direction
	\begin{equation}\label{eq:fine_grid_update_dir}
		d = P_{y}^{x}\big(\invRet_{y}^H(\tilde{z})\big) \in \Tan_{x} \M_h,
	\end{equation}
	together with a~step size $\alpha>0$, defines the new fine iterate
	$x_{\rm new}=\Ret_{x}^h(\alpha d)\in\M_h$. The flowchart of the two-level coarse correction step is presented in Figure~\ref{fig:multilevel_sketch}.
	
	\begin{figure}
		\centering
		\includegraphics{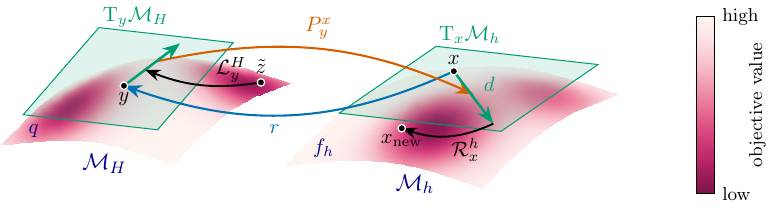}
		\caption{Flowchart of a two-level coarse correction step.} 
		\label{fig:multilevel_sketch}
	\end{figure}
	
	The following proposition provides the Riemannian gradient of the coarse model objective.
	
	\begin{prop}[Riemannian gradient of the coarse model]\label{prop:gradient_coarse_model}
		The Riemannian gradient of the coarse model objective~$q$ in~\eqref{eq:riem_coarse_model} at $z \in \calC_{y}$
		is given by
		\begin{equation}\label{eq:gradPsi}
			\grad q(z) = \grad f_H(z) - \Drm \invRet_{y}^H(z)^\ast (w),
		\end{equation}
		where $ \Drm \invRet_{y}^H(z)^\ast$ denotes the adjoint of $\Drm \invRet_{y}^H(z)$ with respect to the metrics $\langle\cdot,\cdot\rangle_y$ and $\langle\cdot,\cdot\rangle_z$.
	\end{prop}
	\begin{proof}
		For $z \in \calC_{y}$, define $g(z) = \langle w, \invRet_{y}^H(z) \rangle_{y}$.
		Its directional derivative along \mbox{$u \in \Tan_z\M_H$} is given by 
		\[
		\Drm g(z)[u] = \langle w,  \Drm \invRet_{y}^H(z)[u] \rangle_{y} 
		= \langle \Drm \invRet_{y}^H(z)^\ast (w),  u \rangle_{z}. 
		\]
		This implies that $\grad g(z) = \Drm \invRet_{y}^H(z)^\ast (w)$, which immediately yields~\eqref{eq:gradPsi}.
	\end{proof}
	
	The first-order coherence property below ensures that critical points of
	the fine objective~$f_h$ correspond to critical points of the coarse
	model~\eqref{eq:riem_coarse_model}, and is essential for convergence.
	
	\begin{prop}[First-order coherence] \label{prop:coherence}
		For the coarse model \eqref{eq:riem_coarse_model}, we have
		\begin{equation} \label{eq:coherence}
			\grad q(y) = R_{x}^{y} \big(\grad f_h(x)\big).
		\end{equation}
	\end{prop}
	\begin{proof}
		From the definition of the retraction, we have
		$\Drm \invRet^H_{y}(y) = (\Drm \Ret^H_{y}(0))^{-1} = \id_{\Tan_{y}\M_H}$
		and hence $\Drm \invRet^H_{y}(y)^{\ast} = \id_{\Tan_{y}\M_H}$. Proposition~\ref{prop:gradient_coarse_model} then yields \eqref{eq:coherence}.
	\end{proof}
	
	The next proposition shows that any sufficient decrease of the coarse
	model produces a~descent direction for the fine objective.
	
	\begin{prop}[Descent direction]\label{prop:descent_dir} 
		Let $\tilde{z}\in\calC_y$ be an approximate solution to \eqref{eq:riem_coarse_model} such that $q(\tilde{z})<q(y)$ and let $d$ be defined as in~\eqref{eq:fine_grid_update_dir}. If $f_H$ is $\Ret^H$-convex on $\calC_y$, then 
		\begin{equation} \label{eq:descent_direction}
			\langle \grad f_h(x), d\rangle_{x} \le q(\tilde{z}) - q(y) < 0.
		\end{equation}
	\end{prop}
	\begin{proof}
		Using \eqref{eq:riem_galerkin}, \eqref{eq:riem_coarse_model}, and $q(y) = f_H(y)$, we obtain for arbitrary $z \in \calC_y$ that
		\begin{align*}
			q(z) 
			&= f_H(z) - \big\langle \grad f_H(y) - R_{x}^{y}(\grad f_h(x)), \invRet_{y}^H(z)\big\rangle_{y} \\
			&
			= q(y) + \big\langle \grad f_h(x) , P_{y}^{x}(\invRet_{y}^H(z))\big\rangle_{x} - \big(f_H(y) - f_H(z) + \langle  \grad f_H(y), \invRet_{y}^H(z)\rangle_{y}\big)\\
			&\ge q(y) + \big\langle \grad f_h(x) , P_{y}^{x}(\invRet_{y}^H(z))\big\rangle_{x},
		\end{align*}
		where the last inequality follows from Proposition~\ref{prop:R_supp_ineq}. Plugging  $z=\tilde{z}$ into this inequality and using~\eqref{eq:fine_grid_update_dir}, we get~\eqref{eq:descent_direction}.
	\end{proof}
	
	\begin{rem}[Choice of the coarse model objective]\label{rem:coarse_choice}
		The coarse objective $f_H$ in~\eqref{eq:riem_coarse_model} is treated as an~independent discretization of the problem objective on a~coarser mesh, i.e., as a~genuinely \emph{geometric} coarse model, rather than one obtained from $f_h$ via the algebraic pullback $f_H = f_h \circ p$ induced by a~point prolongation map~$p$. This construction provides the flexibility to convexify or regularize $f_H$ to ensure that the resulting coarse correction is a~descent direction {\rm(}cf. Proposition~\textup{\ref{prop:descent_dir}}{\rm)}, a~freedom that is not available in purely algebraic coarse models. In the limiting case $f_H \equiv 0$, the coarse problem~\eqref{eq:riem_coarse_model} reduces to maximizing the linear correction term, thereby recovering the classical Galerkin coarse correction of the linear multigrid method~\textup{\cite{Nash:2000,Trottenberg2001}}.
	\end{rem}
	
	A~further important structural property of the coarse model~\eqref{eq:riem_coarse_model} is its independence of the choice of Riemannian metric on $\M_H$, provided that the vector restriction $R_x^y$ is defined as the metric adjoint of the vector prolongation $P_y^x$. This reflects the intuition that the construction relies only on first-order information, namely the differentials of $f_H$ and $f_h$, which are themselves metric-independent.
	
	\begin{prop}[Metric-independence of the coarse model] \label{prop:metric_independence}
		Let $\langle\cdot, \cdot \rangle_{y,1}$, $\langle\cdot, \cdot \rangle_{y,2}$ and $\langle\cdot, \cdot \rangle_{x,1}$, $\langle\cdot, \cdot \rangle_{x,2}$ be different Riemannian metrics on $\M_H$ and $\M_h$, respectively. Furthermore, for \mbox{$y \in \M_H$} and~$x \in \M_h$, let $P_{y}^{x}$ be a~vector prolongation, and let $R_{x,i}^{y}$, $i = 1,2$, be the adjoint of~$P_{y}^{x}$ with respect to the $i$-th metric, which means that for all $u \in \Tan_{y} \M_H$ and $v\in \Tan_{x} \M_h$, 
		\begin{equation}\label{eq:Galerkin_i}
			\big\langle P_{y}^{x}u, v \big\rangle_{x,i} 
			= \big\langle u, R_{x,i}^{y}v \big\rangle_{y,i} , \qquad i = 1,2.
		\end{equation}
		Then the coarse models \eqref{eq:riem_coarse_model} induced by both metrics coincide and produce the same search direction.
	\end{prop}
	\begin{proof}
		For $i =1,2$, let $\grad_{i} f_H(y)$ and $\grad_{i} f_h(x)$ denote the Riemannian gradients with respect to the corresponding metrics on $\M_H$ and $\M_h$, respectively. Then using~\eqref{eq:Galerkin_i}, the coarse model~\eqref{eq:riem_coarse_model} can be written as 
		\begin{equation*}
			\begin{aligned}
				q(z) &=  f_H(z) - \big\langle \grad_{i} f_H(y)-R_{x,i}^{y}(\grad_{i} f_h(x)), \invRet^H_{y}(z)\big\rangle_{y,i} \\
				&
				= f_H(z)  - \big\langle \grad_{i} f_H(y), \invRet^H_{y}(z)\big\rangle_{y,i} 
				+ \big\langle \grad_{i} f_h(x), P_{y}^{x}(\invRet^H_{y}(z)) \big\rangle_{x,i} \\
				&= f_H(z)  - \Drm f_H(y)[\invRet^H_{y}(z)] + \Drm f_h(x) [P_{y}^{x}(\invRet^H_{y}(z))] .
			\end{aligned}
		\end{equation*}
		Since the last expression is independent of $i$, and this equality holds for both choices of~$i$, the coarse models coincide. Consequently, for the same (approximate) solution $\tilde{z}$, they produce the identical search direction.
	\end{proof}
	
	Proposition~\ref{prop:metric_independence} identifies a key design principle: once $P^x_y$ is fixed, enforcing the Galerkin condition $R^y_x = (P^x_y)^\ast$ yields a~metric-independent coarse model. Thus, the primary degree of freedom lies in the choice of $P^x_y$, which should be guided by problem-specific geometric considerations. It should also be noted that although the coarse model does not dependent on the metric for a~fixed~$P^x_y$, a~prolongation well adapted to the geometry of $\M_h$ and $\M_H$, i.e., their Riemannian metrics, produces better-scaled descent directions.
	In the next section, we develop several structurally distinct approaches to construct vector transfer operators based on point transfer maps.
	
	
	\section{Construction of vector transfer operators}
	\label{sec:vector_transport_choice}
	
	The construction of \emph{vector transfer operators} is dictated by the coarse model~\eqref{eq:riem_coarse_model}, which requires a~consistent exchange of first-order information between discretization levels. Specifically, the coarse model~\eqref{eq:riem_coarse_model} involves restricting the fine-level Riemannian gradient to the coarse tangent space via~$R_x^y$, while the resulting coarse correction is prolongated back via~$P_y^x$ to define a search direction. The construction of these operators is built on point-level transfer maps between the manifolds.
	
	\begin{defn}
		{\em  A smooth surjective map $\restr : \dom(\restr)\subseteq \M_h \to \M_H$ is called a~\emph{point restriction} if its differential $\Drm \restr(x)$ is surjective for all $x\in \dom(\restr)$.}
	\end{defn}
	
	\begin{defn}
		{\em A smooth injective map $\prol : \M_H \to \M_h$ is called a~\emph{point prolongation} if its differential $\Drm \prol(y)$ is injective for all $y \in \M_H$.}
	\end{defn}
	
	Typical choices for point prolongation in the Euclidean setting include interpolation
	i.e., neighborhood-weighted averaging, while point restrictions can be constructed via injection, i.e., subsampling or aggregation-based compression, e.g., full-weighting~\textup{\cite{BriHM00}}. They can be adapted to the geometry of manifolds, e.g., interpolation on the Bernoulli manifold~\textup{\cite{MuePZ23}}.

	Although point restriction and prolongation maps are generally nonlinear, their differentials provide canonical candidates for vector transfer operators: $\Drm r(x)$ for~$R_x^y$ and $\Drm p(y)$ for~$P_y^x$, with the other determined by the geometric Galerkin condition~\eqref{eq:riem_galerkin}, which by 
	Proposition~\ref{prop:descent_dir} ensures a~descent direction. This gives 
	rise to three constructions: two restriction-based 
	(Section~\ref{sec:restriction-based}) and one prolongation-based 
	(Section~\ref{sec:prolongation-based}). When the manifolds admit an ambient 
	embedding, projection-based constructions are also available 
	(Section~\ref{sec:projection-based}). An overview of the different approaches is provided in Figure~\ref{fig:sec4-overview} and Table~\ref{tab:vector_ops} below.
	
	\captionsetup{width=.95\textwidth, skip=3pt}
	\begin{figure}[t]
		\hspace*{6mm}
		\includegraphics{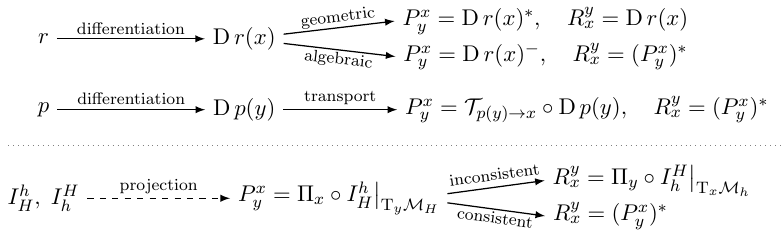}
		\caption{Overview of the vector transfer operator constructions.
			Differential-based vector transfer operators (above dotted line) satisfy the geometric Galerkin condition~\eqref{eq:riem_galerkin} by construction and do not require an~ambient embedding. The projection-based
			approaches (below dotted line) rely on an~ambient embedding, and the geometric Galerkin condition is satisfied only in the consistent variant.}
		\label{fig:sec4-overview}
	\end{figure}
	
	\subsection{Restriction-based approaches} \label{sec:restriction-based}
	
	Both approaches below take $\Drm r(x)$ as the starting point and differ 
	in how the compatible prolongation~$P_y^x$ is derived from it.
	
	\subsubsection{Geometric approach}\label{ssec:r-geom}
	
	For a~point restriction~$r$ and $x \in \dom(r)$, we set $y = r(x)$ and define
	the \emph{restriction-based geometric} vector transfer operators
	\[
	P_y^x = \Drm r(x)^\ast, \qquad\quad R_x^y = \Drm r(x).
	\]
	Since $P_y^x$ and $R_x^y$ are adjoints of each other by construction, the 
	geometric Galerkin condition~\eqref{eq:riem_galerkin} is automatically satisfied.
	The name \emph{geometric} reflects that $P_y^x = \Drm r(x)^\ast$ is defined 
	via the Riemannian metrics on $\M_h$ and~$\M_H$. Since $P_y^x$ changes 
	with these metrics, the assumption of Proposition~\ref{prop:metric_independence} 
	is not satisfied and the coarse model is metric-dependent. This dependence 
	can, however, be exploited: choosing a metric that encodes the structure 
	of the problem induces a~vector prolongation $P_y^x$ that acts as a preconditioner for the coarse correction. This approach extends to arbitrary $x \in \dom(r)$ and $y \in \M_H$ by composing $\Drm r(x)$ with a~vector transport 
	$\calT_{r(x)\to y}\colon \Tan_{r(x)}\M_H \to \Tan_y\M_H$ on the coarse level.
	
	\subsubsection{Algebraic approach}\label{ssec:r-alg}
	For a point restriction $r$ and $x \in \dom(r)$, we consider once again $y = r(x)$ and construct the \emph{restriction-based algebraic} vector transfer operators
	\begin{equation} \label{eq:r_algebraic_prol_restr}
		P_y^x = \Drm \restr(x)^-, \qquad\quad
		R_x^y = (P_y^x)^\ast,    
	\end{equation}
	where $\Drm \restr(x)^-$ denotes a~right inverse of $\Drm \restr(x)$.
	The name \emph{algebraic} reflects that the vector prolongation $P_y^x = \Drm r(x)^-$ is 
	metric-independent, since any right inverse of a linear map can be 
	constructed without reference to a metric. By 
	Proposition~\ref{prop:metric_independence}, the coarse model~\eqref{eq:riem_coarse_model} is therefore also metric-independent. Moreover, this construction is consistent with the special case $f_h = f_H \circ r$, as the following proposition shows.
	
	\begin{prop}\label{prop:algebraic1} Let $x \in \dom(r)$ and $y = \restr(x)$ for a~point restriction $\restr$.
		If $f_h = f_H \circ \restr$ on $\dom(r)$, then for the vector transfer operators \eqref{eq:r_algebraic_prol_restr}, the coarse model objective in~\eqref{eq:riem_coarse_model} satisfies $q(z) = f_H(z)$ for all $z\in \calC_{y}$.
	\end{prop}
	\begin{proof}
		Since $\Drm \restr(x)$ is surjective, its right inverse $P_{y}^{x}$ is injective, and $R_{x}^{y} = (P_{y}^{x})^\ast$ is surjective. 
		By the definition of the right inverse, we obtain that
		\[
		R_{x}^{y} \circ \Drm \restr(x)^\ast 
		= \big(\Drm \restr(x)\circ\Drm \restr(x)^-\big)^\ast
		= \id_{\Tan_{y}\M_H}.
		\]
		By chain rule, we further have 
		\begin{align*}
			\grad f_h(x) &= \grad (f_H \circ r)(x) 
			= \Drm \restr(x)^\ast(\grad f_H(\restr(x))).
		\end{align*}
		Therefore, 
		\[
		w 
		= \grad f_H(y)-\big(R_{x}^{y}\circ \Drm \restr(x)^\ast\big)(\grad f_H(y))
		= 0,
		\]
		and consequently $q(z) = f_H(z)$  for all $z \in \calC_{y}$.
	\end{proof}
	
	\subsection{Prolongation-based approach}\label{sec:prolongation-based}
	The differential $\Drm p(y)$ is a~natural choice for the vector prolongation. However, for an arbitrarily point $x \in \M_h$, $\Drm p(y)$ does not map into $\Tan_x\M_h$ unless $x = p(y)$. To handle this \emph{base-point mismatch}, we compose $\Drm p(y)$ with a~vector transport $\calT_{p(y)\to x}\colon \Tan_{p(y)}\M_h \to \Tan_x\M_h$ and define
	\begin{equation} \label{eq:p_geom_prol_restr}
		P_y^x = \calT_{p(y)\to x} \circ \Drm p(y), \qquad R_x^y = (P_y^x)^\ast.
	\end{equation}
	The geometric Galerkin condition~\eqref{eq:riem_galerkin} is satisfied by construction. Natural choices for $\calT_{p(y)\to x}$ include the identity when $x = p(y)$, or any vector transport along a~curve connecting $p(y)$ to~$x$. In embedded manifolds, the
	orthogonal projection $\Pi_x\colon V_h \to \Tan_x\M_h$ provides a~convenient computational option. When all tangent spaces are canonically identified with~$V_h$~-- as in the box-constrained Bernoulli-manifold setting of~\cite{MuePZ23}, used again in Section~\ref{sec:Binary-CC}, where $\M_h \subset \R^n$ is open 
	and $\Tan_x\M_h = \R^n$ for all $x\in\M_h$
	~-- no correction is needed and 
	$\calT_{p(y)\to x} = \id_{\Tan_x\M_h}$ for all $x$. Since $\Drm p(y)$ does not require an~inner product, $P_y^x$ is metric-independent whenever $\calT_{p(y)\to x}$ is, and by Proposition~\ref{prop:metric_independence}, the coarse model~\eqref{eq:riem_coarse_model} is also metric-independent.
	
	In the special case $x = p(y)$ and $f_H = f_h \circ p$, the coarse model objective reduces exactly to the coarse-level objective, as established in the following proposition.
	
	\begin{prop}\label{prop:algebraic2}
		Let $y \in \M_H$ and $x = p(y)$ for a~point prolongation~$p$. If 
		$f_H = f_h \circ p$ on an~$\Ret^H$-convex set $\calC_y \subseteq \M_H$, 
		then the coarse model objective in~\eqref{eq:riem_coarse_model}, constructed using  the vector transfer operators~\eqref{eq:p_geom_prol_restr} with 
		\mbox{$\calT_{p(y)\to x} = \id_{\Tan_x\M_h}$}, satisfies $q(z) = f_H(z)$ for all 
		$z \in \calC_y$.
	\end{prop}
	\begin{proof}
		Applying the chain rule to $f_H = f_h \circ p$, we obtain 
		\[
		\grad f_H(y) = \Drm p(y)^\ast(\grad f_h(p(y))) = R_x^y(\grad f_h(x)).
		\]
		Then $w = \grad f_H(y) - R_x^y(\grad f_h(x)) = 0$, and consequently 
		$q(z) = f_H(z)$ for all $z \in \calC_y$.
	\end{proof}
	
	\subsection{Projection-based approaches}\label{sec:projection-based}
	
	Unlike the differential-based constructions above, this approach does not 
	rely on a point map $r$ or $p$, but instead uses ambient restriction and prolongation operators combined with tangent projections~\cite{SutV2021}. Specifically, suppose that $\M_h$ and $\M_H$ are embedded in the Euclidean spaces $V_h$ and $V_H$, respectively. Let
	\[
	I_H^h:V_H\to V_h, \qquad\qquad 
	I_h^H:V_h\to V_H
	\] 
	be the prolongation and restriction on the ambient spaces, and let 
	\[
	\Pi_x:V_h \to \Tan_x\M_h, \qquad\quad 
	\Pi_y:V_H \to \Tan_y\M_H
	\]
	denote the orthogonal projections onto the corresponding tangent spaces. We 
	define the \emph{projection-based inconsistent} vector prolongation and restriction as
	\begin{equation} \label{eq:projection_prol_restr}
		P_y^x = \Pi_x \circ I_H^h\big|_{\Tan_y\M_H}, \qquad\quad
		R_x^y = \Pi_y \circ I_h^H\big|_{\Tan_x\M_h},
	\end{equation}
	respectively. Since $\Pi_x$ depends on the Riemannian metric on $\M_h$, $P_y^x$ is metric-dependent. 
	
	In general, the vector transfer operators~\eqref{eq:projection_prol_restr} do not satisfy the geometric Galerkin condition~\eqref{eq:riem_galerkin}. To enforce it, one can keep $P_y^x$ as in~\eqref{eq:projection_prol_restr} and set $R_x^y = (P_y^x)^\ast$; we refer to this as the \emph{projection-based consistent} approach.
	
	In multigrid methods, standard prolongation and restriction operators typically satisfy 
	$I_h^H = (I_H^h)^\ast$ with respect to suitable inner products on 
	$V_h$ and~$V_H$. When, in addition, $\Pi_x$ and~$\Pi_y$ are orthogonal 
	with respect to the metrics on $\M_h$ and~$\M_H$ induced by these inner 
	products, both projection-based constructions coincide.
	
	\subsection{Comparison}
	\label{sec:comparison}
	
	The various vector prolongation and restriction operators are summarized in Table~\ref{tab:vector_ops}.
	By construction, all approaches, except for the projection-based inconsistent one, provide operators that satisfy the geometric Galerkin condition~\eqref{eq:riem_galerkin}.  
	
	\begin{table}[ht]
		\captionsetup{width=0.95\textwidth, skip=6pt}
		\centering
		\scalebox{.88}{
			\begin{tabular}{l l l l l c c l}
				\toprule
				Version & Name & Base  & $P_y^x$ & $R_x^y$ & Galerkin & $P_y^x$ metric- & Ref. \\[-1mm]
				&   & points & & & condition & independ. & \\
				\midrule\midrule
				I  & restriction-based & $x$ arb.      & $\Drm r(x)^\ast$ & $\Drm r(x)$
				& \checkmark & \text{\sffamily X} & new \\[-.51mm]
				& geometric & $y = r(x)$& & & & & \\ 
				\midrule
				II & restriction-based   &  $x$ arb.    & $\Drm r(x)^-$ & $(P_y^x)^\ast$
				& \checkmark & \checkmark & new \\[-.51mm]
				& algebraic  & $y = r(x)$ & & & & & \\
				\midrule
				III & prolongation- & $x, y$ arb. & $\calT_{\hat{x}\to x}\circ\Drm p(y)$ & $(P_y^x)^\ast$
				& \checkmark & \checkmark${}^\dag$ & \cite{MuePZ23} \\[-.51mm]
				& based & $\hat{x} = p(y)$ & & & & & \\
				\midrule
				IV  & projection-based        & $x,y$ arb. & $\Pi_x\circ I_H^h\big|_{\Tan_y\M_H}$
				& $\Pi_y\circ I_h^H\big|_{\Tan_x\M_h}$
				& \text{\sffamily X} & \text{\sffamily X} & \cite{SutV2021} \\[-.51mm]
				& inconsistent  & & & & & & \\
				\midrule
				V   & projection-based   & $x,y$ arb. & $\Pi_x\circ I_H^h\big|_{\Tan_y\M_H}$
				& $(P_y^x)^\ast$
				& \checkmark & \text{\sffamily X} & new \\[-.51mm]
				& consistent  & & & & & & \\
				\bottomrule
			\end{tabular}
		}
		\caption{Vector transfer operator constructions. 
			${}^\dag$$\Drm p(y)$ is metric-free, however $\calT_{\hat{x}\to x}$ may introduce  metric-dependence when $x \neq \hat{x}=p(y)$.}
		\label{tab:vector_ops}
	\end{table}
	
	The vector transfer operators constructed from differentials of the point maps and from tangent space projections differ in how tangent spaces across discretization levels are related. The projection-based approaches depend on the choice of a~Riemannian metric and are less compatible with the multilevel discretization structure. In contrast, the differential-based approaches are naturally aligned with this structure but depend on the choice of point maps, which may introduce discretization-dependent inconsistencies on general manifolds. When the point maps are linear and the Galerkin condition $I_h^H = (I_H^h)^\ast$ holds with respect to the Euclidean inner products, all constructions reduce to the classical linear multigrid operators~$I_H^h$ and~$I_h^H$. On general manifolds, they represent complementary ways of defining vector transfer operators.
	
	Comparing the prolongation-based and restriction-based constructions of vector transfer operators, the two approaches differ in their choice of primary point map. In the prolongation-based setting, a~coarse-to-fine map is specified first, focusing on the reconstruction of the fine-scale structure. In contrast, the restriction-based approach starts from a fine-to-coarse map that compresses fine-scale information, making it more natural for defining a~coarse model and for ensuring consistency between the discretization levels and the induced vector transfer operators.
	
	Within the restriction-based framework, both the geometric and algebraic approaches start from the differential $\Drm \restr$, but differ in how the associated vector prolongations are defined, namely either as an~adjoint or as a~right inverse of~$\Drm \restr$. The geometric construction emphasizes compatibility with the underlying metric, whereas the algebraic construction prioritizes exact recoverability at the linearized level. In particular, the latter yields a~metric-independent vector prolongation $P_y^x$, but this operator is generally not unique. 
	
	
	\section{Multilevel Riemannian optimization}
	\label{sec:geometric-multilevel}
	
	In this section, we build on the Riemannian coarse model introduced in Section~\ref{sec:coarse_model} and the vector transfer operators developed in Section~\ref{sec:vector_transport_choice} to formulate a~multilevel Riemannian optimization algorithm. The proposed algorithm is rooted in the Euclidean multilevel approach of~\cite{Nash:2000,WenGoldfarb2010} and extends the Riemannian multilevel methods of~\cite{SutV2021} for low-rank manifolds and of~\cite{MuePZ23} for box-constrained problems on the Bernoulli manifold, both of which arise as special cases of the present framework. Here, we provide a~unified treatment for general Riemannian manifolds, together with a~convergence analysis that is not available in either of these prior works. We first discuss the two-level variant and then extend it recursively to the multilevel setting.
	
	\subsection{Two-level Riemannian optimization approach}
	\label{ssec:two-level}
	
	Building on the Riemannian coarse model~\eqref{eq:riem_coarse_model} and a
	choice of vector transfer operators~$P_{y_k}^{x_k}$ and~$R_{x_k}^{y_k}$ from Section~\ref{sec:vector_transport_choice}, we now introduce a~two-level Riemannian optimization method for solving~\eqref{eq:min}.
	For a~given point restriction $r$, starting with an~initial guess $x_0\in \dom(r)$, the method proceeds iteratively. At iteration $k$, whenever the conditions for invoking the coarse model are satisfied, the coarse correction step is performed to update the fine-level iterate, followed by the gradient step to refine the solution. As described in Section~\ref{sec:coarse_model}, the coarse correction step consists in restricting the current iterate $x_k\in \M_h$ to a~coarse level, $y_k=\restr(x_k)$, and approximately solving the coarse model
	\begin{equation} \label{eq:riem_coarse_model_k}
		\min \limits_{z \in \calC_{y_k}} q_k(z) \coloneqq f_H(z) - \big\langle w_k, \invRet^H_{y_k}(z)\big\rangle_{y_k} \enskip \text{with }\; w_k = \grad f_H(y_k) - R_{x_k}^{y_k}\big(\grad f_h(x_k)\big)
	\end{equation}
	to obtain an approximate minimizer $\tilde{z}_k \in \calC_{y_k}\subseteq\mathcal{M}_H$ of the 
	coarse model. The coarse correction direction is then extracted via the 
	lifting map as $\invRet_{y_k}^H(\tilde{z}_k) \in \Tan_{y_k}\mathcal{M}_H$, 
	and transferred to the fine level by the vector prolongation to yield the 
	search direction 
	\[
	d_k = P_{y_k}^{x_k}\bigl(\invRet_{y_k}^H(\tilde{z}_k)\bigr) \in \Tan_{x_k}\mathcal{M}_h.
	\]
	The iteration is subsequently updated via $x_{k+1} = \Ret_{x_k}^h(\alpha_kd_k)$ with an~appropriate step size~$\alpha_k>0$. In the gradient step, the new iteration is refined by $x_{k+1} = \Ret_{x_k}^h\big(\!-\alpha_k \grad f_h(x_k)\big)$. This procedure is repeated iteratively until a~prescribed convergence criterion is satisfied, such as a sufficiently small Riemannian gradient norm or a~negligible decrease in the objective value. The resulting two-level Riemannian optimization method is summarized in Algorithm~\ref{alg:two_level_Riemannian}.
	
	\begin{algorithm}[th]
		\SetKwInOut{Input}{input}
		\SetKwRepeat{Repeat}{repeat}{until}
		\SetKwIF{If}{ElseIf}{Else}{if}{then}{else if}{else}{end if}
		\SetAlgoLined
		\DontPrintSemicolon
		\Input{\parbox[t]{\linewidth}}{Manifolds  $\M_h$ and $\M_H$, objectives $f_h$ and $f_H$, retraction~$\Ret^h$, lifting~$\invRet^H$, point restriction map $r$, vector prolongation $P$, vector restriction $R$, and initial point $x_0 \in \dom(r)$.
		}
		\BlankLine
		
		\For{$k=0,1,\ldots$}{
			\uIf(\hfill\tcp*[h]{coarse~correction~step}){coarse model condition \eqref{eq:est_eta} at $x_k$ is satisfied \textbf{{\em and}} {\em(}$k\!=\!0$ \textbf{{\em or}} $d_{k-1}\!=\!-\grad f_h(x_{k-1})${\em)}
			}{
				Compute $y_k = \restr(x_k)$.\;
				Compute an~approximate solution $\tilde{z}_k\! \in \M_H$ to \eqref{eq:riem_coarse_model_k} such that $q_k (\tilde{z}_k) < q_k(y_k)$. \;
				Compute the search direction $d_k = P_{y_k}^{x_k}(\invRet_{y_k}^H(\tilde{z}_k))$.    
			}
			\Else(\hfill\tcp*[h]{gradient step}){Compute the search direction $d_k = -\grad f_h(x_k)$.\;}
			Compute $\alpha_k > 0$ such that $f_h(\Ret_{x_k}^h(\alpha_kd_k)) < f_h(x_k)$ and $\Ret_{x_k}^h(\alpha_kd_k) \in \dom(r)$.\; 
			Update $x_{k+1} = \Ret_{x_k}^h(\alpha_kd_k)$.\;
		}
		\caption{Two-level Riemannian optimization method}
		\label{alg:two_level_Riemannian}
	\end{algorithm}
	
	In line~2 of this algorithm, the condition for accepting the coarse model is given by
	\begin{equation} \label{eq:est_eta}
		\big\|R_{x_k}^{y_k}\big(\grad f_{h}(x_{k})\big)\big\|_{y_k} 
		\geq \max(\eta\big\|\grad f_{h}(x_{k})\big\|_{x_k},\,\mu), \qquad \eta\in(0,1), \;\mu > 0.
	\end{equation}
	This condition is adapted from \cite{WenGoldfarb2010} to the Riemannian setting; see also \cite{MuePZ23}. It ensures that a~coarse correction is performed only when the coarse level retains a~significant amount of first-order information available on the fine level (quantified by $\eta$) and is sufficiently far from stationarity (controlled by $\mu$).
	In addition, we impose the structural condition
	\begin{equation}\label{eq:no_two_C}
		\text{if $d_{k}$ is a~coarse correction step, then } d_{k-1}=-\grad f_h(x_{k-1}),
	\end{equation}
	which forbids two consecutive coarse corrections and ensures that 
	every coarse correction (except at the very first iteration) is preceded by a~gradient step (smoothing) on the fine level. This requirement is incorporated into 
	Algorithm~\ref{alg:two_level_Riemannian} and plays a~key role in the 
	convergence analysis in Section~\ref{ssec:convergence}.
	
	By Proposition~\ref{prop:coherence}, we have $\grad q_k(y_k) = R^{y_k}_{x_k}(\grad f_h(x_k))$, and therefore the left-hand side of~\eqref{eq:est_eta} is precisely the norm of the Riemannian gradient of the coarse model at $y_k$. If this quantity is small, the vector restriction operator has discarded a~substantial part of the relevant descent information from the fine level, making a~coarse correction unlikely to be effective. In such cases, it is preferable to perform a fine-level gradient step instead.
	
	In line~4 of Algorithm~\ref{alg:two_level_Riemannian}, the coarse model~\eqref{eq:riem_coarse_model_k} is typically solved only approximately, as its main purpose is to provide an efficient correction rather than a highly accurate solution. Starting from a~restricted version of the fine-level iterate $y_k=r(x_k)$, one usually applies a~few steps of a~Riemannian optimization method, such as a~(preconditioned) gradient descent, on the coarse manifold $\M_H$. The use of appropriate vector transfer operators ensures that the geometric structure of the problem is respected. Moreover, the efficiency of the coarse model solver can be improved by choosing the Riemannian metric on~$\M_H$ adaptively, for instance by incorporating first- or second-order information from~$f_H$.
	
	To compute the step size $\alpha_k$ in line~10 of Algorithm~\ref{alg:two_level_Riemannian}, one may apply an~exact, Armijo, Wolfe, or Hager-Zhang line search adapted to the Riemannian setting~\cite{AbsMS08,SutV2021,ZhaH04}. More advanced step size strategies, such as a~non-monotone line search algorithm combined with the alternating Barzilai-Borwein technique~\cite{WenY13,ZhaH04} or an~adaptive procedure from~\cite{AnsM25}, require access to Riemannian gradients at two successive iterates. Consequently, they are only applicable to two consecutive gradient steps within the two-level algorithm. Nevertheless, such strategies can still be effectively utilized when solving the coarse model~\eqref{eq:riem_coarse_model_k}.
	
	\subsection{Multilevel approach}
	\label{ssec:miltilevel}
	
	The two-level Riemannian optimization algorithm extends naturally to the multilevel setting by introducing a~hierarchy of discretizations of increasing dimension and applying the two-level procedure recursively across multiple levels. Analogously to multigrid methods for linear systems, e.g.,~\cite{BriHM00}, this leads to \mbox{V-}, W-, F-, or adaptive cycle schemes that balance computational cost and convergence efficiency. In particular, the coarse model is solved recursively on a~hierarchy of successively coarser models. 
	
	For $\ell_{\rm f}\ge2$, let $h_1 > \ldots > h_{\ell_{\rm f}}$ denote a~sequence of mesh sizes corresponding to increasingly finer discretization levels. Associated with each level $\ell=1,\ldots,\ell_{\rm f}$, we consider a~manifold~$\M_{h_\ell}$ and the corresponding objective function $f_{h_\ell}:\M_{h_\ell}\to\R$. Let $\Ret^{h_\ell}$ and $\invRet^{h_\ell}$ denote, respectively, the retraction and lifting maps on $\M_{h_\ell}$. Furthermore, for $x_{k,\ell}\in\M_{h_\ell}$ and a~point restriction map $\restr_\ell : \dom(\restr_\ell) \subseteq \M_{h_\ell} \to \M_{h_{\ell -1}}$, the associated vector restriction operator $R_{x_{k,\ell}}^{y_{k,\ell}}$, where $y_{k,\ell} = \restr_\ell(x_{k,\ell})$, maps tangent vectors from fine to coarse levels, while the vector prolongation operator $P_{y_{k,\ell}}^{x_{k,\ell}}$ transfers corrections from coarse to fine levels. We refer to this multilevel construction as a~\emph{discretization hierarchy} $h_1 > \ldots > h_{\ell_{\rm f}}$.
	
	The multilevel Riemannian optimization method is presented in Algorithm~\ref{alg:restriction_based_multilevel}. It recursively applies the two-level correction across the discretization hierarchy combined with fine-level smoothing. At each level, once accepted, the coarse problem is approximately solved using the same procedure, resulting in a~recursive adaptive cycle-type scheme. This structure allows information to be efficiently propagated between levels, while reducing computational cost. To minimize $f_{h_{\ell_{\rm f}}}$ at the finest level, the entire discretization hierarchy $h_1 > \dots > h_{\ell_{\rm f}}$ is provided to Algorithm~\ref{alg:restriction_based_multilevel}, together with the objective $f^{\ell_{\rm f}} = f_{h_{\ell_{\rm f}}}$ and the initial point $x_{0,\ell_{\rm f}} \in \dom(r_{\ell_{\rm f}})\subseteq\M_{h_{\ell_{\rm f}}}$.
	Here, $f^{\ell}$ denotes the objective supplied to the algorithm at level~$\ell$: at the finest level, $f^{\ell_{\rm f}}=f_{h_{\ell_{\rm f}}}$, while for $\ell<\ell_{\rm f}$, $f^{\ell}$ is the coarse model constructed at the next finer level, built from the independently discretized objective~$f_{h_\ell}$.

	\begin{algorithm}[t!]
		\SetKwInOut{Input}{input}
		\SetKwRepeat{Repeat}{repeat}{until}
		\SetKwIF{If}{ElseIf}{Else}{if}{then}{else if}{else}{end if}
		\SetAlgoLined
		\DontPrintSemicolon
		\Input{\parbox[t]{\linewidth}}{Discretization hierarchy with mesh sizes $h_1>\dots > h_{\ell}$, objective $f^{\ell}$, initial point $x_{0,\ell} \in \dom(r_{\ell})$.
		}
		\BlankLine
		\For{$k=0,1,\ldots$}{
			\uIf(\hfill\tcp*[h]{coarse correction step}){coarse model condition at $x_{k,\ell}$ is satisfied \textbf{{\em and}} {\em(}$k = 0$ \textbf{{\em or}} $d_{k-1,\ell} = -\grad f^{\ell}(x_{k-1,\ell})${\em)}}{
				Compute $y_{k,\ell} = \restr_\ell(x_{k,\ell})$.\;
				Construct the coarse model $q_{k,\ell}(z)=q_k\big(z; x_{k,\ell},y_{k,\ell},  f^{\ell}, f_{h_{\ell-1}}, \invRet^{h_{\ell-1}}_{y_{k,\ell}}, R_{x_{k,\ell}}^{y_{k,\ell}}\big)$
				according to~\eqref{eq:riem_coarse_model_k}.\;
				\uIf(\hfill\tcp*[h]{standard two-level on the coarsest two levels}){$\ell = 2$}{
					Compute an~approximate minimizer $\tilde{z}_{k,\ell}\in \M_{h_{\ell-1}}$ of $q_{k,\ell}$ such that $q_{k,\ell}(\tilde{z}_{k,\ell}) < q_{k,\ell}(y_{k,\ell})$ using any Riemannian optimization method.\;
				}
				\Else(\hfill\tcp*[h]{recursive multilevel on reduced hierarchy}){
					Compute an~approximate minimizer $\tilde{z}_{k,\ell}\in \M_{h_{\ell-1}}$ of $q_{k,\ell}$ such that $q_{k,\ell}(\tilde{z}_{k,\ell}) < q_{k,\ell}(y_{k,\ell})$ using multilevel Riemannian optimization with the discretization hierarchy $h_1>\dots > h_{\ell -1}$, the objective $f^{\ell-1} = q_{k,\ell}$,  and the initial point $x_{0,\ell-1} = y_{k,\ell}$.\;
				}
				Compute the search direction $d_{k,\ell} = P_{y_{k,\ell}}^{x_{k,\ell}}\big(\invRet_{y_{k,\ell}}^{h_{\ell-1}}(\tilde{z}_{k,\ell})\big)$.    
			}
			\Else(\hfill\tcp*[h]{gradient step}){Compute the search direction $d_{k,\ell} = -\grad f^{\ell}(x_{k,\ell})$.\; 
			}
			Compute the step size $\alpha_k > 0$ such that $f^{\ell}\big(\Ret_{x_{k,\ell}}^{h_\ell}(\alpha_k d_{k,\ell})\big) < f^{\ell}(x_{k,\ell})$.\; 
			Update $x_{k+1,\ell} = \Ret_{x_{k,\ell}}^{h_\ell}\big(\alpha_k d_{k,\ell}\big)$.\;
		}
		\caption{Multilevel Riemannian optimization method}
		\label{alg:restriction_based_multilevel}
	\end{algorithm}
	
	\subsection{Convergence analysis}
	\label{ssec:convergence}
	
	We aim now to analyze the convergence of the multilevel Riemannian optimization method. It is sufficient to consider only the two-level formulation, since the multilevel scheme can be viewed as a~two-level method in which the coarse model is solved recursively by the same two-level Riemannian optimization procedure. The resulting convergence analysis is independent of the number of levels, as it relies only on classifying each step as a~gradient or coarse correction step, regardless of how the latter was generated.
	
	Consider the two-level Riemannian optimization algorithm that generates a~sequence \linebreak\mbox{$(x_k)_{k \in \N} \subset \M_h$} via $x_{k+1} = \Ret_{x_k}^h(\alpha_k d_k)$, where each search direction $d_k \in \Tan_{x_k}\M_h$ is one of the following two types:
	\begin{itemize}[leftmargin=2em]
		\item \emph{G-type} (gradient step) if $d_k = -\grad f_h(x_k)$;
		\item \emph{C-type} (coarse correction step) if 
		$d_k \neq -\grad f_h(x_k)$, but $d_k$ is a descent direction, i.e., $\langle \grad f_h(x_k), d_k\rangle_{x_k} < 0$.
	\end{itemize}
	By construction~\eqref{eq:no_two_C}, Algorithm~\ref{alg:two_level_Riemannian} does not permit two consecutive C-type steps. Consequently, for every C-type index $k$, the preceding step satisfies $d_{k-1}=-\grad f_h(x_{k-1})$, while the successor $d_{k+1}$, if generated, is necessarily of G-type. Accordingly, we introduce the index sets
	\[
	\calK_G = \big\{ k\in\N \enskip : \enskip d_k = -\grad f_h(x_k) \big\}, \qquad
	\calK_C = \big\{ k\in\N \enskip : \enskip d_k \neq -\grad f_h(x_k) \big\}.
	\]
	We make the following assumptions on the objective function $f_h$:
	\begin{itemize}[itemsep=0.2em]
		\item[\bf{A1:}] (Sublevel-set compactness) For all $x_0 \in \M_h$, 
		$\mathcal S_{x_0} \coloneqq \{x \in \M_h : f_h(x) \le f_h(x_0)\}$ is compact.
		\item[\bf{A2:}] (Smoothness and lower boundedness) $f_h \in C^1(\M_h)$ and $f_h$ is bounded below on $\M_h$.
		\item[\bf{A3:}] (Lipschitz continuous differentiability) 
		The pullback $f_h\circ \Ret_x^h$ is uniformly Lipschitz continuously differentiable on $\mathcal S_{x_0}$, i.e., there exists $L_{f_h}>0$ for all $x \in \mathcal S_{x_0}$,  
		and all $v_1,v_2\in \Tan_x\M_h$, $\|\nabla f_h(\Ret_x^h(v_1)) - \nabla 
		f_h(\Ret_x^h(v_2))\|_{x} \le L_{f_h}\|v_1-v_2\|_x$, where 
		$\nabla f_h(\Ret_x^h(v))$ denotes the gradient of $f_h\circ \Ret_x^h$ at $v$ with respect to the inner product on~$\Tan_x\M_h$.
	\end{itemize}
	In addition, we impose the following assumptions on the step sizes $\alpha_k$:
	\begin{itemize}[itemsep=0.2em]
		\item[\bf{A4:}]
		The step size $\alpha_k$ satisfies the Riemannian Wolfe conditions
		\begin{subequations}\label{eq:Wolfe}
			\begin{align}
				f_h(x_{k+1}) - f_h(x_k) &\;\leq\; c_1\,\alpha_k\,\big\langle \grad f_h(x_k), d_k\big\rangle_{x_k},
				\label{eq:Wolfe1}\\
				\big\langle \grad f_h(x_{k+1}), \Drm\Ret_{x_k}^h(\alpha_k d_k)[d_k]\big\rangle_{x_{k+1}}
				&\;\geq\; c_2\,\big\langle \grad f_h(x_k), d_k\big\rangle_{x_k},
				\label{eq:Wolfe2}
			\end{align}
		\end{subequations}
		where $0 < c_1 < c_2 < 1$,  and $\Drm\Ret_{x_k}^h(\alpha_k d_k)$ denotes the differentiated retraction vector transport.
		
		\item[\bf{A5:}] There exists $\alpha_{\max}>0$ such that $\alpha_k \le \alpha_{\max}$ for all $k\in\calK_G$.
	\end{itemize}
	Under these assumptions, we establish the global convergence of the two-level Riemannian optimization algorithm. 
	
	\begin{thm}\label{thm:conv}
		Let Assumptions~\textup{\textbf{A1}--\textbf{A5}} be fulfilled. Then the sequence $(x_k)_{k\in\N}$ generated by the two-level Riemannian optimization algorithm satisfies
		\begin{equation}\label{eq:conv_grad}
			\lim_{k \to \infty} \|\grad f_h(x_k)\|_{x_k} = 0.
		\end{equation}
	\end{thm}
	\begin{proof}
		We may assume $\grad f_h(x_k) \neq 0$ for all $k\in\N$, since otherwise the algorithm terminates at a~stationary point and there is nothing to prove. Then, because each $d_k$ is a~descent direction and Assumptions~\textbf{A2}--\textbf{A4} hold,
		the Riemannian Zoutendijk theorem~\cite[Thm.~2]{RiWi2012} yields
		\begin{equation}\label{eq:zout}
			\sum_{k=0}^{\infty}
			\frac{\langle \grad f_h(x_k), d_k\rangle_{x_k}^{2}}{\|d_k\|_{x_k}^{2}}
			\;<\; \infty.
		\end{equation}
		Since the summands in~\eqref{eq:zout} are non-negative, the series may be decomposed into two complementary convergent subseries corresponding to the G-type and C-type iterations:
		\begin{equation}\label{eq:split}
			\sum_{k\,\in\,\calK_G}
			\frac{\langle \grad f_h(x_k), d_k\rangle_{x_k}^{2}}{\|d_k\|_{x_k}^{2}}
			\;+\;
			\sum_{k\,\in\,\calK_C}
			\frac{\langle \grad f_h(x_k), d_k\rangle_{x_k}^{2}}{\|d_k\|_{x_k}^{2}}
			\;<\; \infty.
		\end{equation}
		For every $k\in\calK_G$, we have $d_k = -\grad f_h(x_k)$, and hence 
		the G-part of \eqref{eq:split} reduces to 
		\begin{equation}\label{eq:Gsum}
			\sum_{k\in\calK_G} \|\grad f_h(x_k)\|_{x_k}^{2} \;<\; \infty.
		\end{equation}
		Since the terms in~\eqref{eq:Gsum} are non-negative and the series converges, it follows that 
		\begin{equation}\label{eq:Gconv}
			\|\grad f_h(x_k)\|_{x_k} \rightarrow 0\qquad \text{as } \;k\rightarrow\infty, \; k\in\calK_G.
		\end{equation}
		
		We next prove that
		\begin{equation}\label{eq:Cconv}
			\|\grad f_h(x_k)\|_{x_k} \rightarrow 0\qquad \text{as } \;k\rightarrow\infty, \; k\in\calK_C.
		\end{equation}
		To this end, we argue by contradiction. Suppose, to the contrary, that this convergence fails on the C-type index set $\calK_C$. Then there is a subsequence $(k_l)_{l \in \N}\subset \calK_C$ and an~$\varepsilon > 0$ such that
		\begin{equation}\label{eq:contra}
			\|\grad f_h(x_{k_l})\|_{x_{k_l}} \geq \varepsilon
			\qquad\text{for all } l \in \N.
		\end{equation}
		As two consecutive C-type steps are forbidden, every $k_l$ is preceded by a
		G-type index. Define
		\[
		\hat{x}_l \coloneqq x_{k_l - 1}
		\qquad\text{so that}\qquad
		d_{k_l - 1} = -\grad f_h(\hat{x}_l)
		\quad\text{(G-type predecessor).}
		\]
		The Wolfe condition \eqref{eq:Wolfe1} implies $f_h(x_{k+1}) \le f_h(x_k)$ for all $k\in\N$, and hence all $\hat{x}_l$ lie in the compact sublevel set from Assumption~\textbf{A1}. By compactness, there exists a~subsequence $(\hat{x}_{l_j})_{j \in \N}$ and a~point $\hat{x}_\infty \in \M_h$ such that 
		\begin{equation}\label{eq:yinf}
			\hat{x}_{l_j} \rightarrow \hat{x}_\infty \qquad\text{as } \;j \rightarrow \infty.
		\end{equation}
		Because $\hat{x}_l = x_{k_l - 1}$ are all G-type iterates, \eqref{eq:Gconv} implies that $\|\grad f_h(\hat{x}_l)\|_{\hat{x}_l} \rightarrow 0$ as $l \rightarrow \infty$.
		In particular, along the subsequence we have
		$\|\grad f_h(\hat{x}_{l_j})\|_{\hat{x}_{l_j}} \rightarrow 0$ as $j \rightarrow \infty$.
		By continuity of $\grad f_h$ and of the Riemannian metric, it follows that 
		\begin{equation}\label{eq:gradzero}
			\|\grad f_h(\hat{x}_\infty)\|_{\hat{x}_\infty} = 0.
		\end{equation}
		
		Further, consider the iterate at the C-type index from the subsequence 
		\[
		x_{k_{l_j}} 
		= \Ret_{\hat{x}_{l_j}}^h\!\bigl(\beta_j \,d_{k_{l_j}-1}\bigr)
		= \Ret_{\hat{x}_{l_j}}^h\!\bigl(-\beta_j \,\grad f_h(\hat{x}_{l_j})\bigr),
		\]
		where $\beta_j = \alpha_{k_{l_j}-1}$ denotes the G-type step size.
		By Assumption~\textbf{A5}, $\beta_j \le \alpha_{\max}$, and therefore
		\begin{equation}\label{eq:tangentnorm}
			\|-\beta_j\,\grad f_h(\hat{x}_{l_j})\|_{\hat{x}_{l_j}}
			\le \alpha_{\max}\,\|\grad f_h(\hat{x}_{l_j})\|_{\hat{x}_{l_j}}
			\rightarrow0 \qquad\text{as } \;j \rightarrow \infty.
		\end{equation}
		Using \eqref{eq:yinf} and~\eqref{eq:tangentnorm}, we conclude that $(\hat{x}_{l_j}, -\beta_j\,\grad f_h(\hat{x}_{l_j}))$ converges in the tangent bundle~$\Tan\M_h$ to $(\hat{x}_\infty, 0)$. By continuity of the retraction $\Ret^h$, it follows that
		\begin{equation}\label{eq:zlim}
			\Ret_{\hat{x}_{l_j}}^h\!\bigl(-\beta_j\,\grad f_h(\hat{x}_{l_j})\bigr)
			\rightarrow \Ret_{\hat{x}_\infty}^h(0) = \hat{x}_\infty \qquad\text{as } \; j \rightarrow \infty.
		\end{equation}
		Combining \eqref{eq:zlim} with continuity of $\grad f_h$ and of the Riemannian metric, and using~\eqref{eq:gradzero}, we obtain that
		\[
		\|\grad f_h(x_{k_{l_j}})\|_{x_{k_{l_j}}}
		\rightarrow \|\grad f_h(\hat{x}_\infty)\|_{\hat{x}_\infty}
		= 0 \qquad\text{as } \; j \rightarrow \infty.
		\]
		However, $\big(x_{k_{l_j}}\big)_{j\in\N}$ is a subsequence of 
		$(x_{k_l})_{l\in\N}$, on which~\eqref{eq:contra} enforces
		\mbox{$\|\grad f_h(x_{k_{l_j}})\|_{x_{k_{l_j}}} \!\ge \varepsilon$} for all $j\in\N$, yielding a~contradiction. Thus, \eqref{eq:Cconv} holds.
		
		Finally, \eqref{eq:Gconv} and \eqref{eq:Cconv}, together with the fact that the index sets $\calK_G$ and $\calK_C$ partition~$\N$, imply~\eqref{eq:conv_grad}.
	\end{proof}
	
	Theorem~\ref{thm:conv} shows that the Riemannian gradient norms vanish along the entire sequence $(x_k)_{k\in\N}$, despite the alternating structure of G- and C-type iterations. In particular, the method ensures first-order stationarity in the limit, so that every accumulation point of~$(x_k)$ is a~critical point of $f_h$.
	
	The convergence proof requires every C-type direction to be a~descent direction for $f_h$. By Proposition~\ref{prop:descent_dir}, this condition is satisfied whenever the coarse objective $f_H$ is locally \mbox{$\Ret^H$-con}\-vex on the candidate set $\calC_{y_k}$ in a~neighborhood of the current coarse iterate~$y_k$, a~property that can often be ensured through the convexification flexibility discussed in Remark~\ref{rem:coarse_choice}. 
	In the absence of local $\Ret^H$-convexity, the algorithm guarantees descent by replacing any non-descent \mbox{C-type} direction $d_k$ with $-\grad f_h(x_k)$. Consequently, Theorem~\ref{thm:conv} remains applicable without imposing any additional assumptions on~$f_H$.
	
	
	\section{Applications}
	\label{sec:applications}
	
	In this section, we turn to concrete applications to illustrate the multilevel Riemannian optimization framework developed so far. Our focus is on energy minimization problems with a~common variational structure, arising in nonlinear quantum models such as the Kohn--Sham and Gross--Pitaevskii equations, as well as in binary image segmentation, formulated here as a~continuous cuts problem. Despite their shared structure, these problems differ substantially in both their analytical properties and the geometry of their constraint manifolds. The Kohn--Sham and Gross--Pitaevskii models are subject to quadratic normalization constraints, giving rise to the Stiefel and ellipsoid manifolds~-- compact and without boundary. In contrast, the continuous cuts problem uses a box constraint, yielding the Bernoulli manifold~-- non-compact, flat, and endowed with a~Fisher--Rao metric that is singular at the boundary. Together, these test examples are well suited for assessing the proposed multilevel framework across a~broad range of geometric regimes and discretization strategies. 
	
	For each application, we begin with a brief description of the underlying model and its variational formulation. We then outline the spatial discretization strategies used in the numerical approximation. Subsequently, we collect the geometric concepts that play a~central role in the numerical treatment, with particular emphasis on the associated manifold structure, the construction of point transfer maps and vector transfer operators, and the resulting coarse models. Finally, we present the results of numerical experiments that demonstrate the performance and qualitative behavior of the multilevel Riemannian optimization schemes. 
	
	Different applications were implemented using programming languages and software tools best suited to their respective models and discretizations. This enabled efficient use of existing libraries and computational frameworks. The source code is publicly available at
	\begin{center}
		\url{https://github.com/Riemannian-Multilevel}
	\end{center}
	
	
	\subsection{Kohn--Sham problem}
	Let $\Omega = \mathbb R^3$ and consider a molecule with $m\in \mathbb N$ electrons. For orbitals $\varphi = (\varphi_1, \dots , \varphi_m) \in [H^1(\Omega,\C)]^m$, the Kohn--Sham energy minimization problem is given by
	\begin{equation}\label{eq:dis_minE_KS}
		\hspace*{-1mm}\begin{array}{l}
			\displaystyle{    \min \calE^{\KS}(\varphi) 
				= \frac{1}{2} \sum\limits_{i = 1}^m \!\int_\Omega \!\|\nabla \varphi_i\|^2\dxi 
				\!+\!\! \int_\Omega\!\vartheta_{\rm n}(\xi) \rho_\varphi\dxi 
				\!+\! \frac{1}{2}\! \int_{\Omega}\!
				\vartheta_{\rm H} (\rho_\varphi)\rho_\varphi \dxi
				\!+\!\! \int_{\Omega}\! \epsilon_{\rm xc}(\rho_\varphi)\rho_\varphi \dxi}\\[4mm]
			\text{subject to } \langle \varphi_i, \varphi_j\rangle_{L^2(\Omega,\C)}= \delta_{ij},\enskip i,j=1,\ldots, m,
		\end{array}
	\end{equation}
	with the electron density $\rho_\varphi = \sum_{i = 1}^m |\varphi_i|^2$. The total Kohn--Sham energy $\calE^{\KS}$ consists of 
	the kinetic energy, 
	the electron–nucleus interaction energy with the nuclear potential $\vartheta_{\rm n}$, 
	the mean-field electron-electron interaction energy, where the Hartree potential generated by the density is given by 
	\[
	(\vartheta_{\rm H}(\rho_\varphi))(\xi) = \int_\Omega \frac{\rho_\varphi(s)}{\|\xi-s\|}\ds,
	\]
	and the exchange-correlation energy in the local density approximation expressed through the exchange-correlation energy per particle $\epsilon_{\rm xc}$. The orthonormality constraints on the components of $\varphi$ mean that each orbital represents an~independent quantum state with unit probability.
	
	Using Euler--Lagrange calculus, constrained critical points of $\,\calE^{\KS}$ can be characterized by the solution to the nonlinear eigenvalue problem (NEVP)
	\begin{equation}\label{eq:nlevp_KS}
		\begin{aligned}
			\calA_{\varphi}^{\KS}\varphi = \varphi\, \Lambda, \qquad
			\big[\langle \varphi_i, \varphi_j\rangle_{L^2(\Omega,\C)}\big]_{i,j=1}^{p,p} = I_p,
		\end{aligned}
	\end{equation}
	with the Kohn--Sham Hamiltonian 
	\[
	\calA_{\varphi}^{\KS} = - \frac{1}{2}\Delta  + \vartheta_{\rm n}  + \vartheta_{\rm H}(\rho_\varphi)  + \vartheta_{\rm xc}(\rho_\varphi),
	\]
	where \mbox{$\vartheta_{\rm xc}(\rho_\varphi)=\frac{\partial}{\partial\rho_\varphi}(\epsilon_{\rm xc}(\rho_\varphi)\rho_\varphi)$} is the exchange-correlation potential. 
	Specifically, the eigenvalues of the Hermitian Lagrange multiplier $\Lambda\in\C^{p\times p}$ correspond to the eigenvalues of $\calA_{\varphi}^{\KS}$. Common approaches for solving the Kohn--Sham energy minimization problem~\eqref{eq:dis_minE_KS} or the NEVP~\eqref{eq:nlevp_KS} include self-consistent field (SCF) iterations \cite{CanKL21,Roo51} and Riemannian optimization techniques \cite{AltPS22,AltPS24,PetPS25,SchRNB09}, treating~\eqref{eq:dis_minE_KS} as a~problem on the Stiefel or Grassmann manifold.
	
	\subsubsection{Spatial discretization by a plane wave method}
	
	For spatial discretization of the Kohn--Sham energy minimization problem \eqref{eq:dis_minE_KS} on a bounded domain $\Omega \subset\R^3$ with periodic boundary conditions, we use the plane wave method on a~Cartesian grid $n_1\times n_2 \times n_3$. Both the spatial and frequency domain discretizations have $n = n_1n_2n_3$ degrees of freedom. In practice, the number of degrees of freedom for the frequency domain is often reduced by a~cutoff energy~$E_{\rm cut}$ for the kinetic energy of the frequency modes, resulting in a number $\le n$ but living on the same spatial grid. This allows us to choose the coarsening factor between levels and the number of levels arbitrarily. The discrete counterpart of \eqref{eq:dis_minE_KS} is given by
	\begin{equation}\label{eq:minE_KS}
		\begin{array}{l}
			\displaystyle{    \min E^{\rm KS}(\Phi) 
				= \frac{1}{2} \trace (\Phi^{\ast} L\, \Phi) + \trace( \Phi^{\ast}  F V_{\rm n} F^{\ast} \Phi)  + 2 \pi \omega P_\Phi^\top F^{\ast}L^+FP_\Phi 
				+ \omega P_\Phi^\top\epsilon_{\rm xc}(P_{\Phi})
			} \\[4mm]
			\text{subject to } \Phi^{\ast} \Phi = I_p,
		\end{array}
	\end{equation}
	where $\Phi\in\C^{n\times m}$ denotes the plane-wave coefficient matrix of the Kohn--Sham orbitals, 
	$\Phi^{\ast}$ is transpose and complex conjugate of $\Phi$, 
	$P_\Phi  = \frac{1}{\omega}\mathrm{diag}\big(F^{\ast}\Phi (F^{\ast}\Phi)^{\ast}\big) \in \R^n$ with $\omega = |\Omega|/n$ is the vector of electron density values on the spatial grid,  
	$F$ denotes the unitary discrete Fourier transform mapping function from the spatial domain to plane-wave coefficients, and~$F^{\ast}$ is its inverse. 
	Furthermore, the diagonal matrix $L$ represents the spectral discretization of the negative Laplacian in the plane-wave basis, $L^+$ denotes its Moore–Penrose pseudoinverse, 
	$2\pi \omega\, F^{\ast}L^+F$ corresponds to the discrete Coulomb operator, 
	$V_{\rm n}$ is the diagonal matrix containing the pointwise evaluation of~$\vartheta_{\rm n}$ on the spatial grid, 
	and the exchange correlation energy $\epsilon_{\rm xc}$ is evaluated componentwise. The discrete Kohn--Sham Hamiltonian then reads
	\[
	A_{\Phi} = \frac{1}{2}L +  FV_{\rm n}F^{\ast} + 4\pi \mathrm{Diag}\big(F^{\ast} L^+F P_\Phi\big) 
	+\mathrm{Diag}\big(\vartheta_{\rm xc}(P_\Phi)\big),
	\]
	where $\mathrm{Diag}(v)$ denotes the diagonal matrix with entries of a vector $v$ on the diagonal. 
	
	The feasibility set for the discrete Kohn--Sham energy minimization problem \eqref{eq:minE_KS} is given by the complex Stiefel matrix manifold
	\[
	\St(m,n)=\big\{ \Phi\in\C^{n\times m} \enskip :\enskip \Phi^{\ast}\Phi=I_m\big\}.
	\]
	Geometric concepts related to this manifold are summarized in Table~\ref{tab:Stiefel}. They can be derived by straightforward extension of the real case, e.g.,~\cite{ShuA23}.
	
	\begin{table}
		\begin{tabular}{lll}
			\toprule 
			complex Stiefel manifold & $\St(m,n)=\big\{ \Phi\in\C^{n\times m} \enskip :\enskip \Phi^{\ast}\Phi=I _m\big\}$
			\tabularnewline
			\midrule 
			tangent space & $\Tan_\Phi\,\St(m,n) = \big\{ V\in\C^{n\times m} \enskip :\enskip V^{\ast}\Phi+\Phi^{\ast}V=0 \big\}$
			\tabularnewline
			\midrule 
			Frobenius metric & $\langle U,V\rangle_{\Phi, \rm F} = \Re e\big(\trace(U^{\ast} V)\big)$, \; $U,V\in\Tan_{\Phi}\,\St(m,n)$ 
			\tabularnewline
			\midrule 
			orthogonal projection & $\Pi_{\Phi}(Y)=I-\tfrac{1}{2}\Phi(\Phi^{\ast}Y+Y^\ast \Phi)$ 
			\tabularnewline
			\midrule 
			Frobenius Riemannian gradient & $\grad_{\rm F} E^{\KS}(\phi) = 2(A_\Phi\,\Phi - \Phi(\Phi^{\ast} A_\Phi\,\Phi))$
			\tabularnewline
			\midrule 
			polar retraction & $\Ret_\Phi(V)=(\Phi+V)(I+V^{\ast}V)^{-1/2}$\tabularnewline
			\midrule 
			polar lifting map & $\invRet_\Phi(Z)=ZY_Z-\Phi\;$ with $\;(\Phi^{\ast}Z)Y_Z+Y_Z(\Phi^{\ast}Z)^{\ast}=2I_m$
			\tabularnewline
			\midrule 
			differentiated lifting map& $\Drm\invRet_\Phi(Z)[U]=UY_Z-Z\hat{Y}_U$ 
			\tabularnewline
			& with $\;(\Phi^{\ast}Z)\hat{Y}_U+\hat{Y}_U(\Phi^{\ast}Z)^{\ast}=\Phi^{\ast}UY_Z+Y_ZU^{\ast}\Phi$
			\tabularnewline
			\bottomrule
		\end{tabular}
		\caption{Geometric concepts for the complex Stiefel manifold $\St(m,n)$.}
		\label{tab:Stiefel}
	\end{table}
	
	\subsubsection{Vector transfer operators}
	
	Let $\St(m,n_H)$ and $\St(m,n_h)$ be the Stiefel manifolds corresponding to the coarse and fine discretizations defined by a~lower and higher cutoff energies~$E_{\rm cut}^H$ and~$E_{\rm cut}^h$, respectively. Let \mbox{$I_H^h:\C^{n_H\times m}\rightarrow\C^{n_h\times m}$} denote the zero-padding operator that maps the plane-wave coefficients on the coarse grid to their low-frequency coefficients on the fine grid, while all new high-frequency coefficients are set to zero. The corresponding restriction operator $I_h^H:\C^{n_h\times m}\rightarrow\C^{n_H\times m}$ is chosen as 
	$I_h^H=(I_H^h)^{\ast}$, where the adjoint is taken with respect to the (real) Frobenius inner products in $\C^{n_h\times m}$ and $\C^{n_H\times m}$. This operator is a~simple truncation that retains only low-frequency information. We define the point restriction map as
	\[
	\restr(\Phi)
	=\big(\pi_{\St(m,n_H)}\!\circ I_h^H\big)(\Phi), 
	\qquad\Phi\in\St(m,n_h), 
	\]
	where $\pi_{\St(m,n_H)}(Y)=Y(Y^{\ast}Y)^{-1/2}$ denotes the Frobenius-norm projection of $Y\in\C^{n_H\times m}$ onto the Stiefel manifold $\St(m,n_H)$. Note that this projection is defined only on a~set of full-rank matrices, which restricts the domain of the point restriction $\restr$ to 
	\[
	\dom(\restr)=\big\{\Phi\in\St(m,n_h)\enskip :\enskip \rank(I_h^H\Phi)=m\big\}.
	\]
	
	The derivative of~$r$ at $\Phi$ along $V\in\Tan_{\Phi}\St(m,n_h)$ is given by 
	\[
	\Drm \restr(\Phi)[V]  =\big(I_h^H V - (I_h^H \Phi)\big((I_h^H \Phi)^{\ast}I_h^H \Phi\big)^{-1/2}X_V \big)\big((I_h^H \Phi)^{\ast}I_h^H \Phi\big)^{-1/2}, 
	\]
	where $X_V$ solves the Lyapunov equation
	\[ 
	\big((I_h^H \Phi)^{\ast}I_h^H \Phi\big)^{1/2}X_V + X_V \big((I_h^H \Phi)^{\ast}I_h^H \Phi\big)^{1/2} = (I_h^H \Phi)^{\ast}I_h^H V + (I_h^H V)^{\ast}I_h^H \Phi.
	\]
	Note that for $\Phi\in\dom(r)$, 
	the matrix $(I_h^H \Phi)^{\ast}I_h^H \Phi$ is Hermitian positive definite, and hence this equation is uniquely solvable. One right inverse of $\Drm \restr(\Phi)$ has the form  
	\[
	(\Drm \restr(\Phi))^-(U) =  I_H^h U \big((I_h^H \Phi)^{\ast}I_h^H \Phi\big)^{1/2}, \qquad U\in\Tan_{r(\Phi)} \St(m,n_H).
	\]
	Defining the algebraic vector prolongation operator
	$P_{\restr(\Phi)}^{\Phi}=\Drm \restr(\Phi)^-$, we determine the corresponding vector restriction operator
	\begin{equation}\label{eq:Restr_KS}
		R_\Phi^{\restr(\Phi)}(V) 
		= \big((\Drm \restr(\Phi))^-\big)^{\ast}(V)
		= \Pi_{r(\Phi)} \big(I_h^H V \big((I_h^H \Phi)^{\ast}I_h^H \Phi\big)^{1/2}\big),
	\end{equation}
	where $\Pi_{r(\Phi)}$ is the orthogonal projector onto $\Tan_{r(\Phi)}\St(m,n_H)$ defined in Table~\ref{tab:Stiefel}.
	
	\subsubsection{Riemannian coarse model} Consider the iterate $\Phi_k\in\St(m,n_h)$ and the restricted iterate $\Psi_k = \restr(\Phi_k)\in\St(m,n_H)$.
	The polar lifting map is defined as 
	\begin{equation}\label{eq:Lift_KS}
		\invRet_{\Psi_k}^H (Z) = Z Y_Z - \Psi_k,
	\end{equation}
	and the adjoint of its derivative is given by 
	\[
	\big(\Drm \invRet_{\Psi_k}^H(Z)\big)^{\ast}(U)
	= \Pi_Z\big((U-\Psi_k X_{Z,U})Y_Z\big),
	\]
	where $\Pi_Z$ is the orthogonal projector onto $\Tan_Z\St(m,n_H)$, and $Y_Z$ and $X_{Z,U}$ are solutions to the Lyapunov equations
	\begin{align}
		(\Psi_k^{\ast}Z)\; Y_Z \enskip + \enskip Y_Z \;(\Psi_k^{\ast}Z)^{\ast} & = 2I_m, \label{eq:lyap1}\\
		(\Psi_k^{\ast} Z)^{\ast} X_{Z,U} + X_{Z,U} (\Psi_k^{\ast}Z) & = Z^{\ast}U+U^{\ast}Z, \label{eq:lyap2}
	\end{align}
	respectively. Note that when $Z=\Psi_k$, these equations have unique solutions. By continuity, if $Z$ is sufficiently close to $\Psi_k$ so that $\|\Psi_k^*Z-I\|_{\rm F}<1$, then the real parts of the eigenvalues of $\Psi_k^*Z$ remain positive, ensuring that both \eqref{eq:lyap1} and \eqref{eq:lyap2} also admit unique solutions. Consequently, the Riemannian coarse model for the Kohn--Sham problem takes the form 
	\[
	\min_{Z\in \calC_{\Psi_k}} q_k^{\rm KS}(Z) = E^{\rm KS}
	_H(Z) - \big\langle \grad_{\rm F} E^{\rm KS}_H(\Psi_k) - R_{\Phi_k}^{\Psi_k} (\grad_{\rm F} E^{\rm KS}_h(\Phi_k)) , \invRet^H_{\Psi_k}(Z) \big\rangle_{\Psi_k, \rm F}
	\]
	with the vector restriction $R_{\Phi_k}^{\Psi_k}$ and  the lifting $\invRet^H_{\Psi_k}$ defined in \eqref{eq:Restr_KS} and \eqref{eq:Lift_KS}, respectively. 
	
	\subsubsection{Numerical experiments}
	
	The numerical experiments for the Kohn--Sham problem were conducted in \texttt{julia} using the density functional toolkit DFTK.jl~\cite{DFTKjcon} for solid state materials. 
	We consider a~gallium arsenide (GaAs) crystal on a~periodic lattice in the faced-centered cubic phase with lattice constant \mbox{$a = 10.68$}~Bohrs. The model employs the local density approximation of the exchange-correlation energy, spin $(\uparrow,\downarrow)$-pairs for the valence electrons, and semicore Goedecker--Teter--Hutter pseudopotentials to represent the core electrons of the gallium and arsenic atoms. 
	A~$4\times4\times4$ Monkhorst--Pack $k$-point grid is used to discretize the Brillouin zone. The discretization parameter for the plane-wave basis is the cutoff energy $E_{\rm cut}$, which gives an~upper bound for the kinetic energy. The cutoff energies for the different discretization levels, together with the resulting numbers of degrees of freedom, are reported in Table~\ref{tab:ndofs}.
	
	\begin{table}[th]
		\centering
		\scalebox{.98}{
			\begin{tabular}{c|c|c|c|c|c|c|c}
				\mbox{$E_{\rm cut}\,(\mathrm{Ha})$} &  
				10 & 
				16 & 
				25 & 
				40 & 
				63 & 
				101 & 
				160 \\\hline 
				\#dofs & 
				18,552 & 
				37,312 & 
				72,576 & 
				147,516 & 
				291,372 & 
				590,512 & 
				1,178,556
			\end{tabular}
		}
		\caption{GaAs model: number of degrees of freedom for the different cutoff energies.}
		\label{tab:ndofs}
	\end{table}
	
	\begin{table}[t]
		\centering
		\scalebox{.98}{
			\begin{tabular}{c|c}
				scheme & $E_{\rm cut}$ \\ \hline
				2-level & 10, 160 \\
				3-level & 10, 40, 160 \\
				4-level & 10, 25, 63, 160 \\
				7-level & \quad 10, 16, 25, 40, 63, 101, 160
			\end{tabular}
		}
		\caption{GaAs model: multilevel setup for different numbers of levels.}
		\label{tab:levels}
	\end{table}
	
	To carry out fine-level updates, we employ a~preconditioned Riemannian gradient descent method based on the $H^1$-metric  
	\[
	\langle U,V\rangle_{\Phi, H^1} = \Re e\big(\trace(U^{\ast}\!(L+I)V)\big), \qquad U,V\in\Tan_\Phi\,\St(m,n),
	\]
	where a shifted negative Laplacian acts as a~preconditioner for the Riemannian gradients. Subsequently, we compare several multilevel optimization approaches with 2, 3, 4, and~7 levels, as detailed in Table~\ref{tab:levels}, the \mbox{$H^1$-Rie}\-man\-nian gradient descent (H1RGD), for which the step size strategy is adopted from~\cite{AltPS24}, and the $H^1$-Riemannian conjugate gradient method (H1RCG), for which the step size strategy and conjugate gradient parameters are adopted from~\cite{PetPS25}. 
	
	In all methods, iterations are terminated once the Frobenius norm of the residual 
	\[
	\res(\Phi_k)
	= \grad_{\rm F} E^{\KS}(\Phi_k)
	= A_{\Phi_k}\Phi_k -\Phi_k(\Phi_k^*A_{\Phi_k}\Phi_k^{})
	\]
	falls below the tolerance~$\tol = 10^{-8}$. The reference energy $E_{\rm ref}$ is computed using SCF with a~tolerance of $10^{-12}$. Within the multilevel schemes, the coarse models are solved using nested multilevel schemes based on H1RGD with a~tolerance of $\max \{10^{-2} \|\res(\Phi_k)\|_{\rm F}, 10^{-8} \}$, except for the coarsest level, where H1RCG is employed. Note that, for brevity, we consider only the restriction-based algebraic approach for the construction of the vector prolongation and restriction (Version~II), as alternative strategies produced almost identical results\footnote{An experiment demonstrating this is provided in the GitHub repository.}.
	
	\begin{figure}[t]
		\includegraphics{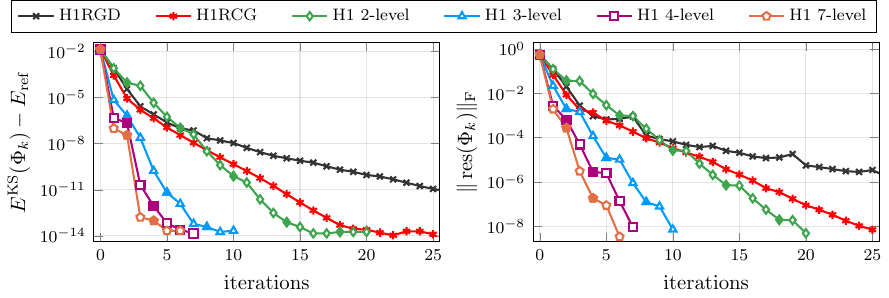}
		\caption{GaAs model: convergence histories of the energy error (left) and residual norm (right) versus iteration count for the different optimization schemes. Filled markers indicate coarse correction steps. All schemes exhibit linear convergence. H1RCG converges faster than H1RGD, while  multilevel methods achieve further acceleration, with improved convergence as the number of levels increases.} 
		\label{fig:GaAs_iter}
	\end{figure}
	
	The convergence plots for the energy error $E^{\KS}(\Phi_k)-E_{\rm ref}$ and the residual norm $\|\res(\Phi_k)\|_{\rm F}$ versus iteration count for the tested methods  are shown in Figure~\ref{fig:GaAs_iter}. One observes that the multilevel methods clearly outperform the single-level H1RGD and H1RCG. Comparing the convergence behavior of the multilevel schemes, we find that a~larger number of levels results in faster convergence. In particular, the coarse model condition~\eqref{eq:est_eta} with $\eta=0.4$ and $\mu=10^{-12}$, as used in our experiments, leads to an~adaptive multilevel cycle strategy. On the finest level, this results in an~alternating sequence of coarse corrections and gradient steps, as indicated in Figure~\ref{fig:GaAs_iter} by filled and hollow markers, respectively. For~7 levels, these steps alternate strictly, while for 4~levels, one to two gradient steps follow, for 3~levels, two gradient steps follow each coarse correction, and for~2~levels, three such steps are performed. The number of smoothing steps depends on the extent to which the error needs to be smoothed such that it is dominated by low-frequency modes that can be resolved on the coarser level, as illustrated in Figure~\ref{fig:ks-error-grid}. Since this level is finer when more levels are used, fewer smoothing steps are necessary. 
	
	Figure~\ref{fig:multilevel_structure} illustrates the detailed multilevel structure of the adaptive cycle for the H1 4-level method, where the pattern emerges naturally from the coarse condition~\eqref{eq:est_eta} rather than being imposed a~priori. By contrast, our experiments showed that the enforcement of fixed patterns that differ from those induced by the coarse condition~\eqref{eq:est_eta} leads to slower convergence.
	
	\begin{figure}[t]
		\centering
		\scalebox{.75}{\includegraphics{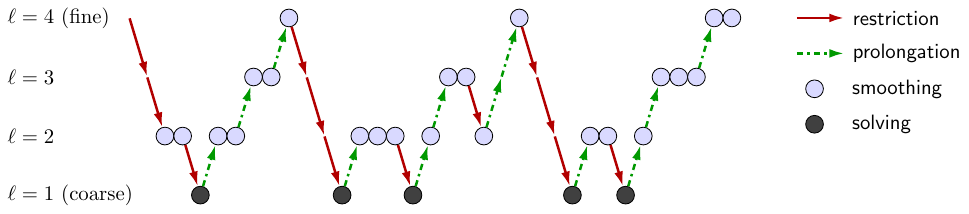}
		}
		\caption{GaAs model: multilevel structure of the adaptive cycle for the H1 4-level method driven by the coarse condition~\eqref{eq:est_eta}. Here, {\sf restriction} corresponds to the construction of a~Riemannian coarse model, while {\sf prolongation} transfers the coarse-level search direction to the fine level followed by a~line search. Furthermore, {\sf smoothing} consists of a single gradient step and {\sf solving} corresponds to applying H1RCG to the~coarsest model until the prescribed residual tolerance is achieved. }
		\label{fig:multilevel_structure}
	\end{figure}
	
	\begin{figure}[t]
		\includegraphics{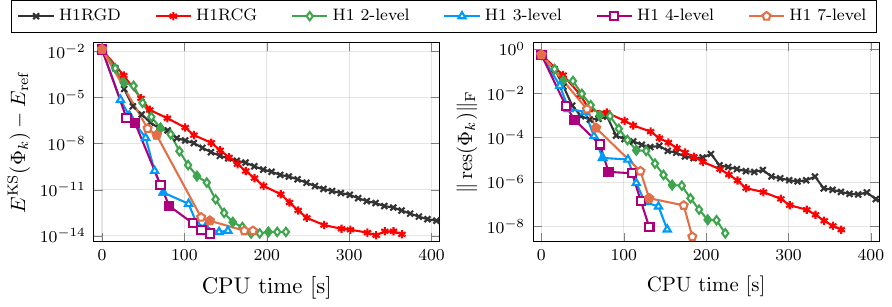}
		\caption{GaAs model: convergence histories of the energy error (left) and residual norm (right) versus CPU time for the different optimization schemes. Filled markers indicate coarse correction steps.
			The multilevel methods achieve substantial time savings over the single-level H1RGD and H1RCG, with the 4-level method performing best among the multilevel schemes.} 
		\label{fig:GaAs_times}
	\end{figure}
	
	It should, however, be noted that while methods with an~increasing number of levels need fewer iterations to converge, coarse correction steps incur additional computational cost. This is illustrated in Figure~\ref{fig:GaAs_times}, which presents the convergence behavior with respect to CPU time. In addition, Table~\ref{tab:times_ks} reports the total computational time together with the percentage of time gained or lost relative to the single-level H1RCG and H1RGD. It can be seen that, despite the additional overhead introduced by the coarse correction steps, the multilevel optimization algorithms still achieve significant speed-ups compared to the single-level methods. Among the tested variants, the 4-level scheme, where the number of degrees of freedom between successive levels differs by approximately a~factor of~4, yields the best overall performance, reducing the runtime of H1RGD by~77\% and of H1RCG by~65\%. 
	
	Finally, Figure~\ref{fig:ks-error-grid} shows the evolution of the error in the electron density along a~slice through the periodic GaAs lattice over the first seven iterates of the single-level H1RCG and the \mbox{H1~4-le}vel variant. It can be observed that coarse corrections
	at iterations 3 and 5 generate errors that are dominated by high-frequency oscillations, while subsequent gradient steps smooth out these errors until the main error contribution is again concentrated in the low-frequency modes, triggering another coarse correction.
	
	\captionsetup{width=0.95\textwidth, skip=3pt}
	\begin{table}[t]
		\centering
		\scalebox{.98}{
			\begin{tabular}{l | c c c c c c}
				& H1RGD   & H1RCG   & H1 2-level & H1 3-level & H1 4-level & H1 7-level \\ \hline
				CPU time  & 594     &  395    & 244        & 167        & \textbf{138}          & 188        \\
				H1RCG     & $+50$\% &         & $-38$\%    & $-58$\%    & $-65$\% & $-52$\%    \\
				H1RGD     &         & $-34$\% & $-59$\%    & $-72$\%    & $-77$\% & $-68$\%  
			\end{tabular}
		}
		\caption{GaAs model: CPU time [s] for the different optimization algorithms and the percentage of time gained or lost relative to the single-level H1RCG and H1RGD.}
		\label{tab:times_ks}
	\end{table}
	
	\begin{figure}
		\hspace{-4mm}\includegraphics{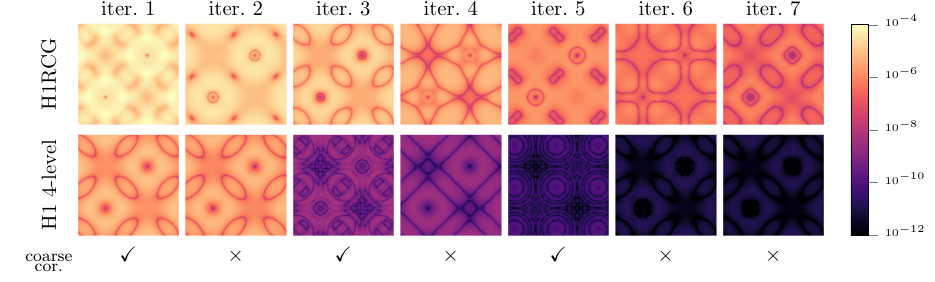}
		\caption{Errors in the electron density along a~slice through the periodic GaAs lattice evaluated at the first seven iterations of the single-level H1RCG and 4-level algorithms. For the 4-level variant, the bottom row indicates whether the respective iteration was a~coarse correction step. }
		\label{fig:ks-error-grid}
	\end{figure}
	
	
	\subsection{Gross--Pitaevskii problem}\label{sec:GrossPitaevskii}
	Unlike the Kohn--Sham model with multiple orbitals, the Gross--Pitaevskii ground-state problem Bose-Einstein condensates involves only a~single (normalized) wavefunction. However, its finite element discretization introduces additional numerical challenges related to mesh construction, adaptivity, and problem geometry, together with the need for effective preconditioners that account for both stiffness and mass matrices. 
	
	Let $\Omega\subset\R^d$ with $d=1,2,3$ be a~bounded convex Lipschitz domain. For a~quantum state $\varphi\in H_0^1(\Omega, \R)$ of a Bose--Einstein condensate, we consider the Gross--Pitaevskii energy minimization problem
	\begin{equation}\label{eq:minE_GP}
		\begin{array}{l}
			\displaystyle{\min \,\calE^{\GP}(\varphi)
				= \int_{\Omega} \frac 12\|\nabla \varphi(\xi)\|^2
				+ \frac 12 \vartheta(\xi)\, |\varphi(\xi)|^2  
				+ \frac\kappa4\, |\varphi(\xi)|^4 \dxi} \\[4mm]
			\text{subject to } \|\varphi\|_{L^2(\Omega,\R)}^2=1,
		\end{array}
	\end{equation}
	where $\vartheta\!\in\! L^\infty(\Omega,\R_+)$ is the external trapping potential confining the system, and~\mbox{$\kappa\in\R$} characterizes the strength of particle interactions. The ground state, defined as global mi\-ni\-mizers of the energy functional $\,\calE^{\GP}$ under the mass constraint, represents the most stable configuration of the condensate. 
	
	Using Euler--Lagrange calculus yields the first-order necessary optimality 
	\begin{equation}\label{eq:gp_nevp}
		\calA_\varphi^{\GP} \varphi 
		= \lambda\,\varphi, \qquad \|\varphi\|_{L^2(\Omega,\R)}^2=1,
	\end{equation}
	with the Gross--Pitaevskii Hamiltonian $\calA_\varphi^{\GP} = -\Delta + \vartheta + \kappa|\varphi|^2$, and the chemical potential \mbox{$\lambda \in \R$} representing the eigenvalue of $\calA^{\GP}_\varphi$. 
	For solving the Gross--Pitaevskii energy minimization problem~\eqref{eq:minE_GP}, different methods have been developed including
	various variants of Sobolev gradient flows~\cite{BaoD04,ChenLLZ24,HenP20},
	Newton-type techniques~\cite{DuL22,WuWB17}, 
	and Riemannian optimization methods~\cite{AltPS22,DanP17,HenY25,HerSY25} in single and multicomponent settings. Efficient implementations often exploit multigrid \cite{BorH08,XuXXF21} or preconditioning techniques \cite{AntLT17,FenT25} to accelerate convergence, especially in rotational or strongly interacting regimes.
	
	\subsubsection{Finite element discretization}
	For spatial discretization of the constrained minimization problem~\eqref{eq:minE_GP},
	we use the finite element discretization with $n$ degrees of freedom. The discretized version of \eqref{eq:minE_GP} takes then the form
	\begin{equation}\label{eq:dis_minE_GP}
		\begin{array}{l}
			\displaystyle{\min E^{\rm GP}(\phi)
				=  \phi^\top \, \Big(\frac{1}{2}\, L 
				+ \frac{1}{2}\, M_{\vartheta} 
				+ \frac{\kappa}{4}\, M_{\phi\phi}\Big)\, \phi} \\[4mm]
			\text{subject to } \phi^\top M\phi=1,		
		\end{array}
	\end{equation}
	where $\phi\in\R^n$ is the discrete quantum state, $L\in\R^{n\times n}$ is the discrete negative Laplacian, $M\in\R^{n\times n}$ is the $L^2$-mass matrix, and \mbox{$M_\vartheta, M_{\phi\phi}\in\R^{n\times n}$} are the weighted mass matrices, where $\phi\phi$ should be understood as the componentwise product. The discrete counterpart of the NEVP~\eqref{eq:gp_nevp}
	is $A_\phi\, \phi=\lambda M\phi$ with the stiffness matrix $A_\phi = L+M_\vartheta + \kappa M_{\phi\phi}$.
	
	The feasible set in the finite-dimensional formulation~\eqref{eq:dis_minE_GP} is given by an ellipsoid 
	\[
	\calS_{M}=\big\{ \phi\in\mathbb{R}^{n}\enskip:\enskip \phi^{\top}M\phi=1\big\}. 
	\]
	The necessary geometric concepts on this manifold are collected in Table~\ref{tab:ellipsoid}. Their derivation can be found in \cite[Example~8.1.4]{AbsMS08} and \cite{Hua13}.
	
	\begin{table}
		\begin{tabular}{lll}
			\toprule 
			manifold & $\calS_{M}=\big\{ \phi\in\mathbb{R}^{n}\enskip:\enskip \phi^{\top}M\phi=1\big\}$
			\tabularnewline
			\midrule 
			tangent space & $\Tan_\phi\,\calS_M = \big\{ v\in\R^n \enskip :\enskip \phi^\top Mv=0 \big\}$
			\tabularnewline
			\midrule 
			$M$-metric & $\langle u,v\rangle_{M} = u^\top M\,v$, \; $u,v\in\Tan_\phi\,\calS_M$ \tabularnewline
			\midrule 
			$M$-orthogonal projection & $\Pi_{\phi,M}=I-\phi\phi^\top M$ 
			\tabularnewline
			\midrule 
			$M$-Riemannian gradient & $\grad_{M} E^{\GP}(\phi) = M^{-1}\big(A_\phi\,\phi - (\phi^\top A_\phi\,\phi)M\phi\big)$
			\tabularnewline
			\midrule 
			$A$-metric & $\langle u,v\rangle_{\phi,A} = u^\top A_{\phi}\,v$, \; $u,v\in\Tan_\phi\,\calS_M$ 
			\tabularnewline
			\midrule 
			$A$-orthogonal projection & $\Pi_{\phi,A}=I-\frac{1}{\phi^\top MA_\phi^{-1} M\phi} A_\phi^{-1}M\phi\phi^\top M$ 
			\tabularnewline
			\midrule 
			$A$-Riemannian gradient & $\grad_{A} E^{\GP}(\phi) = \phi-\frac{1}{\phi^\top MA_\phi^{-1} M\phi} A_\phi^{-1}M\phi$
			\tabularnewline
			\midrule 
			projective retraction 
			& $\Ret_\phi(v)=\frac{\phi+v\enskip}{\|\phi+v\|_{M}}$\tabularnewline
			\midrule 
			projective lifting map
			& $\invRet_\phi(z)=\frac{1}{\phi^{\top}Mz}z-\phi$\tabularnewline
			\midrule 
			differentiated lifting map
			& $\Drm\invRet_\phi(z)[u]=\frac{1}{\phi^\top Mz}\Big(I-\frac{z\,\phi^\top M}{\phi^\top Mz}\Big)u$\tabularnewline
			\bottomrule
		\end{tabular}
		\caption{Geometric concepts for the ellipsoid $\calS_M$.}
		\label{tab:ellipsoid}
	\end{table}
	
	Consider the projection operator mapping the ambient space $\R^n$ onto the ellipsoid $\calS_M$, which is defined as 
	\[
	\pi_{\calS_M}(y) = \argmin_{\tilde\phi\,\in\,\calS_M} \|\tilde\phi-y\|_{M}
	=\frac{y\enskip\,}{\|y\|_{M}}, \qquad y\in\R^n\setminus \{0\}.
	\]
	It induces the \emph{projective retraction}
	\[
	\Ret_\phi(v)
	=\argmin_{\tilde\phi\in\,\calS_M}\|\tilde\phi-(\phi+v)\|_M
	= \pi_{\calS_M}(\phi+v)
	= \frac{\phi+v\enskip\;}{\|\phi+v\|_{M}}, \quad v\in\Tan_\phi\calS_M.
	\]
	The derivative of $\pi_{\calS_M}$ at $y$ is given by 
	\[
	\Drm\pi_{\calS_M}(y)=\frac{1}{\|y\|_M} \left(I-\frac{yy^\top M}{\|y\|_M^2}\right).
	\]
	For $\phi\in\calS_M$, $\Pi_{\phi,M}:=\Drm \pi_{\calS_M}(\phi)=I-\phi\phi^\top M$ is the orthogonal projection onto $\Tan_\phi \,\calS_M$ with respect to the $M$-metric defined in Table~\ref{tab:ellipsoid}. 
	
	\subsubsection{Vector transfer operators} 
	\label{sssec:GP_VectorTransports}
	
	Our goal is now to introduce the point restriction and prolongation maps and the corresponding vector transfer operators using different approaches presented in Section~\ref{sec:vector_transport_choice} based on the $M$-metric.
	
	Let $\calS_{M_H}$ and $\calS_{M_h}$ be the ellipsoids corresponding to the coarse and fine discretizations, respectively, with the mass matrices $M_H\in\R^{n_H\times n_H}$ and $M_h\in\R^{n_h\times n_h}$. For simplicity, we write $\|\cdot\|_M$ for the norm induced by the mass matrix on the corresponding discretization level, omitting the level index whenever it is clear from the context. The same convention applies to the orthogonal projections and Riemannian gradients with respect to the corresponding $M$-metric. Using the interpolation and restriction operators \mbox{$I_H^h:\mathbb{R}^{n_H}\rightarrow\mathbb{R}^{n_h}$} and \mbox{$I_h^H:\mathbb{R}^{n_h}\rightarrow\mathbb{R}^{n_H}$}, respectively, we define the point restriction and prolongation as the compositions 
	\begin{align}
		\restr(\phi)
		& =\big(\pi_{\calS_{M_H}}\!\circ I_h^H\big)(\phi)
		=\frac{I_h^H\phi\enskip\;}{\|I_h^H\phi\|_{M}}, 
		& \hspace*{-15mm}\phi\in\calS_{M_h}, \label{eq:restrGP}\\
		\prol(\psi)
		& =\big(\pi_{\calS_{M_h}}\!\circ I_H^h\big)(\psi)
		= \frac{I_H^h\psi\enskip\;}{\|I_H^h\psi\|_{M}}, 
		& \hspace*{-14mm}\psi\in\calS_{M_H}, \notag
	\end{align}
	provided $I_h^H\phi\neq 0$ and $I_H^h\psi\neq 0$. Their derivatives are given by 
	\begin{align*}
		\Drm \restr(\phi) & =\frac{1}{\|I_h^H \phi\|_{M}}\bigg(I-\frac{(I_h^H \phi)(I_h^H \phi)^{\top}M_{H}}{\|I_h^H \phi\|_{M}^{2}}\bigg)I_h^H,\notag \\
		\Drm \prol(\psi) & =\frac{1}{\|I_H^h \psi\|_{M}}\bigg(I-\frac{(I_H^h \psi)(I_H^h \psi)^{\top}M_{h}}{\|I_H^h \psi\|_{M}^{2}}\bigg) I_H^h.
	\end{align*}
	Using these representations, we define the following vector transfer operators:
	\begin{itemize}
		\item Version III: 
		$P_\psi^{\phi} = \Pi_{\phi,M}\Drm \prol(\psi)$ and 
		\begin{align*}
			\qquad\qquad\qquad\qquad R_{\phi}^{\psi} =  (P_\psi^{\phi})^{\ast}_M 
			= \frac{1}{\|I_H^h \psi\|_M} M_{H}^{-1}(I_H^h)^\top M_{h} \bigg(I-\frac{(I_H^h \psi)(I_H^h \psi)^{\top}M_{h}}{\|I_H^h \psi\|_M^2}\bigg)\bigg|_{\Tan_{\phi}\calS_{M_h}};
		\end{align*}
		
		\item Version IV: 
		$P_{\psi}^{\phi}=\Pi_{\phi, M} I_{H}^{h}\big|_{\Tan_{\psi}\calS_{M_H}}$ and $R_{\phi}^{\psi} =\Pi_{\psi,M} I_{h}^{H}\big|_{\Tan_{\phi}\calS_{M_H}}$;
		
		\item Version V: 
		$P_{\psi}^{\phi} =\Pi_{\phi,M}I_{H}^{h}\big|_{\Tan_{\psi}\calS_{M_H}}$ and 
		\begin{align*}
			R_{\phi}^{\psi} &=(P_{\psi}^{\phi})^{\ast}_{M} = (I-\psi\psi^{\top}M_{H})M_H^{-1}(I_{H}^{h})^{\top}M_{h}\big|_{\Tan_{\phi}\calS_{M_h}}; \quad
		\end{align*}
	\end{itemize}
	We do not consider the other vector transfer operators presented in Section~\ref{sec:vector_transport_choice}, as they are computationally more expensive. For example, the restriction based geometric vector prolongation (Version~I) involves an~additional inversion on the fine level. 
	
	\subsubsection{Mass-weighed Riemannian coarse model}
	\label{subsec:geom-coarse}
	Next, we provide the expressions for the Riemannian coarse model with respect to the $M$-metric and its Riemannian gradients. 
	Let $\phi_{k}\in\mathcal{S}_{M_h}$ and $\psi_k=r(\phi_{k})\in\mathcal{S}_{M_H}$ with the point restriction map $\restr$ defined in~\eqref{eq:restrGP}. The coarse model with respect to the $M$-metric is given by
	\begin{align*}
		\min_{z\in\calC_{\psi_k}}q_{k}^{\GP}(z) 
		& =E_{H}^{\GP}(z)-\big\langle w_{k},\invRet_{\psi_k}^H(z)\big\rangle_{M_H},  \\
		w_{k} & =\grad_M E_{H}^{\GP}(\psi_k) - R_{\phi_{k}}^{\psi_k}\big(\grad_M E_{h}^{\GP}(\phi_{k})\big). 
		\notag
	\end{align*}
	To solve this minimization problem on the coarsest level and to carry out the fine-grid gradient steps, we use the energy-adaptive Riemannian gradient decent (EARGD) method \cite{HenP20} based on the energy-adaptive metric (shortly, $A$-metric) as defined in Table~\ref{tab:ellipsoid}. The $A$-Riemannian gradient of the coarse model objective $q_k^{\GP}$ is given by  
	\[
	\grad_{A} q_k^{\GP}(z) 
	= \Pi_{z,A}\left(z - \frac{1}{\psi_k^\top M_{H}z} A_z^{-1}M_H\Big(I- \frac{\psi_k\,z^\top M_H}{\psi_k^\top M_Hz}\Big) w_k\right).
	\]
	Note that its computation involves two inversion of $A_{z}$ on the coarse level.
	
	\begin{table}[t]
		\centering\vspace*{-3mm}
		\scalebox{.98}{
			\begin{tabular}{c|c|c|c|c}
				$\ell_{\rm ref}$ &  
				8 &  
				9 &  
				10 &  
				11   
				\\\hline 
				\#dofs 
				& 16,441
				& 66,049  
				& 263,169
				& 1,050,625 \\\multicolumn{5}{}{}\\\multicolumn{5}{}{}
			\end{tabular}
		}
		\hspace{5mm}
		\scalebox{.98}{\begin{tabular}{c|c}
				scheme & $\ell_{\rm ref}$ \\ \hline
				2-level & 10, 11 \\
				3-level & 9, 10, 11 \\
				4-level & 8, 9, 10, 11 
		\end{tabular}}
		\caption{Gross--Pitaevskii model: number of degrees of freedom for different refinement levels (left) and multilevel setup for different numbers of levels (right).}
		\label{tab:dof_GP}
	\end{table}
	
	\subsubsection{Numerical experiments}
	\label{sec:experiments}
	
	The Gross--Pitaevskii problem is discretized using the C++ finite element library \texttt{deal.ii} \cite{ABBBFGHHKMMSTUWW25}. 
	The experimental setup is as follows. We consider the Gross--Pitaevskii model with homogeneous Dirichlet boundary conditions on the spatial domain $\Omega=[-11,11]^2$. The external trapping potential is given by the harmonically confined optical lattice  
	\[
	\vartheta(\xi)=\frac{1}{2}\big(\xi_1^2+\xi_2^2\big) 
	+ 100 \left( \sin^2\left(\tfrac{\pi\xi_1}{2}\right)+\sin^2\left(\tfrac{\pi\xi_2}{2}\right) \right),
	\]
	and the interaction strength is set to $\kappa=1000$, which corresponds to a~strongly interacting and computationally challenging regime. For the spatial discretization, we employ bilinear finite elements on a~hierarchy of quadrilateral meshes of mesh width $h_{\ell_{\rm ref}}=10\cdot 2^{-(\ell_{\rm ref}-1)}$, where $\ell_{\rm ref}\in\{8,9,10,11\}$ denotes the refinement level.  
	The corresponding numbers of degrees of freedom on the different refinement levels together with the multilevel setup for different numbers of levels are reported in Table~\ref{tab:dof_GP}.
	The transfer operators between the refinement levels in the finite element spaces are realized using nodal injection for $I_h^H$ and bilinear nodal interpolation for $I_H^h$, as provided by \texttt{deal.ii}. 
	All linear systems are solved using the conjugate gradient (CG) method with a~maximum of $500$ iterations. For systems involving the mass matrix $M$, a~fixed residual tolerance of $10^{-6}$ is used. For systems involving the stiffness matrix $A_{\phi_k}$, we employ an~adaptive tolerance of $10^{-2} \|\res(\phi_k)\|_{M}$, where 
	\[
	\res(\phi_k)=A_{\phi_k}\phi_k - (\phi_k^\top A_{\phi_k}\phi_k)M\phi_k
	\]
	denotes the residual associated with the fine-level iterate $\phi_k$. Furthermore, for the coarse model $q_k^\GP$, we define the residual corresponding to an~iterate~$z_l$~by
	\begin{align*}
		\mathrm{qres}(z_l)=M_H\grad_{M}q_{k}^{\GP}(z_l) 
		& =A_{z_l}z_l-(z_l^\top \!A_{z_l}z_l)M_Hz_l
		- \frac{1}{\psi_k^\top M_{H}z_l}\Big(I - \frac{\psi_k\, z_l^\top\! M_H}{\psi_k^\top M_{H}z_l}\Big)w_{k}.
	\end{align*}
	When solving linear systems involving $A_{z_l}$ on the coarse levels, the CG iteration is run until the norm of the linear-system residual is reduced below the adaptive tolerance of $10^{-2}\|\mathrm{qres}(z_l)\|_M$. 
	The EARGD iterations are terminated after $4$ steps on the coarse levels and once the residual norm satisfies $\|\res(\phi_k)\|_M\leq 10^{-8}$ on the finest level. The reference ground state $\phi_{\rm ref}$ and the reference energy $E_{\rm ref}$ are computed by running EARGD until the residual norm reaches $10^{-15}$.
	The step sizes are determined using the Armijo backtracking line search. In the coarse model condition~\eqref{eq:est_eta}, the parameters are set to~$\eta=0.5$ and~$\mu=10^{-4}$. Note that once the threshold $\mu=10^{-4}$ is reached, no further coarse correction steps are performed for the remainder of Algorithm~\ref{alg:restriction_based_multilevel}.
	
	\captionsetup{width=.95\textwidth, skip=3pt}
	\begin{figure}[t]
		\centering
		\includegraphics{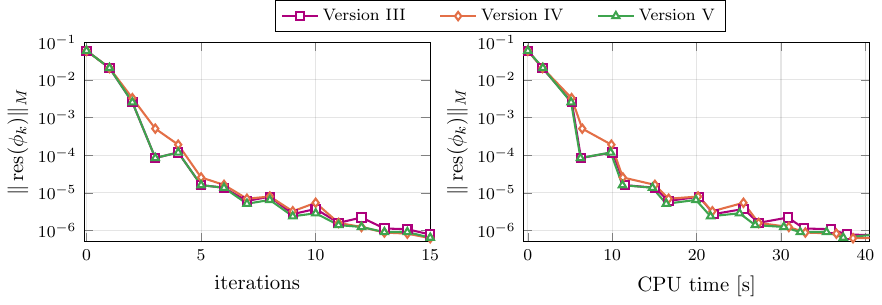}
		\caption{Gross--Pitaevskii model: convergence histories of the residual norm versus iteration count (left) and CPU time (right) for the 4-level EARGD scheme with different vector transfer operators.}
		\label{fig:gp_4level}
	\end{figure}
	
	\captionsetup{width=.95\textwidth, skip=4pt}
	\begin{figure}[t]
		\centering
		\includegraphics{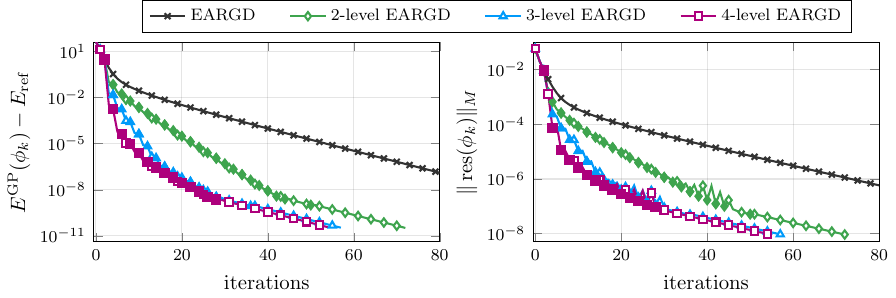}
		\caption{Gross--Pitaevskii model: convergence histories of the energy error (left) and residual norm (right) versus iteration count for different optimization schemes. Filled markers indicate coarse correction steps. The multilevel methods clearly outperform the single-level EARGD, with the \mbox{4-le}vel variant achieving the fastest convergence.} 
		\label{fig:gp_iter2}
	\end{figure}
	
	\captionsetup{width=.95\textwidth, skip=4pt}
	\begin{figure}[t!]
		\centering
		\includegraphics{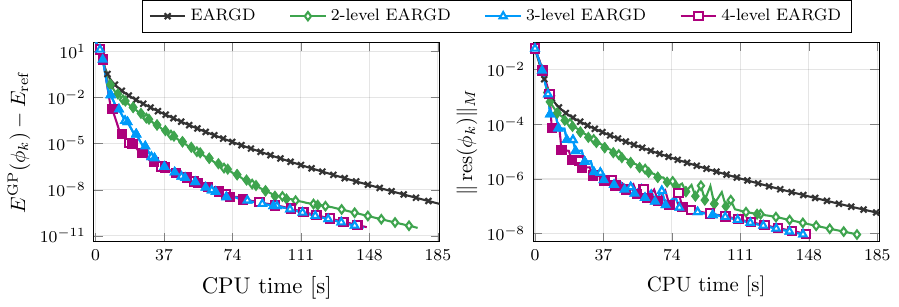}
		\caption{Gross--Pitaevskii model: convergence histories of the energy error (left) and residual norm (right) versus CPU time for different optimization schemes. Filled markers indicate coarse correction steps. The multilevel methods yield significant reductions in computational time compared to the single-level EARGD, with the 3-level variant being the most efficient. 
		} 
		\label{fig:gp_time2}
	\end{figure}
	
	In Figure~\ref{fig:gp_4level}, we compare the residual norms $\|\mathrm{res}(\phi_k)\|_M$ versus iteration count and CPU time for the 4-level EARGD  scheme with different vector transfer operators defined in Section~\ref{sssec:GP_VectorTransports}. We present the results for the first 15 iterations only, during which the algorithm alternates between coarse correction and gradient steps. For further details on the iteration structure, we refer to Figures~\ref{fig:gp_iter2} and~\ref{fig:gp_time2}. It can be observed that all tested operators exhibit very similar performance, with the prolongation-based (Version~III) and projection-based consistent (Version~V) operators performing slightly better than the projection-based inconsistent (Version~IV) operators. This highlights the importance of satisfying the geometric Galerkin condition. 
	
	\captionsetup{width=0.95\textwidth, skip=3pt}
	\begin{table}[t]
		\centering
		\scalebox{.98}{
			\begin{tabular}{c|ccccc}
				& EARGD & 2-level EARGD  & 3-level EARGD  & 4-level EARGD  
				\tabularnewline
				\hline 
				CPU time & 237 & 174 & \textbf{145} & 147 
				\tabularnewline
				EARGD &  & $-27\%$ & $-39\%$ & $-39\%$ 
				\tabularnewline
		\end{tabular}}
		\caption{Gross--Pitaevskii model: CPU time [s] for the different optimization algorithms and the percentage of time gained relative to the single-level EARGD.}
		\label{tab:times_gp}
	\end{table}
	
	\begin{figure}[t]
		\centering
		\scalebox{.75}{\includegraphics{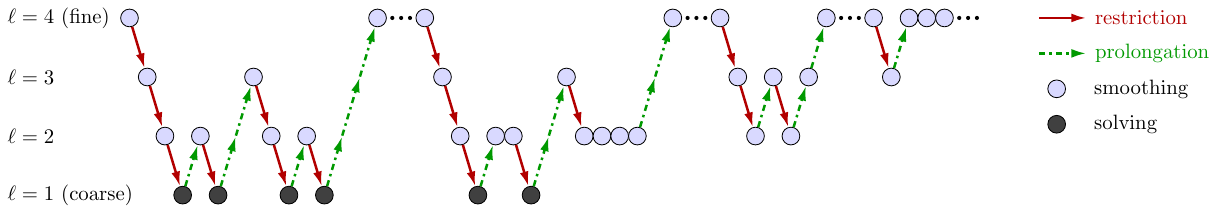}
		}
		\caption{Gross--Pitaevskii model: multilevel structure of the adaptive cycle for the 4-level EARGD method driven by the coarse condition~\eqref{eq:est_eta}. Here, {\sf restriction} corresponds to the construction of a~Riemannian coarse model, while {\sf prolongation} transfers the coarse-level search direction to the fine level followed by a~line search. Furthermore, {\sf smoothing} consists of a single gradient step and {\sf solving} corresponds to applying EARGD to the~coarsest model until the prescribed residual tolerance is achieved.}
		\label{fig:multilevel_structure_gp}
	\end{figure}
	
	In the subsequent experiments, for simplicity of presentation, we use the prolongation-based (Version~III) vector transfer operators throughout, as they provide representative performance. Figures~\ref{fig:gp_iter2} and~\ref{fig:gp_time2} present the energy errors and residual norms for different optimization schemes with various multilevel hierarchies. The numerical results clearly demonstrate the advantages of the multilevel strategies over the single-level approach. Their improved convergence behavior is reflected not only in the iteration counts but also in the overall computational efficiency. Thereby, the 4-level EARGD  scheme exhibits the best convergence performance while being only slightly more expensive than the 3-level variant; see Table~\ref{tab:times_gp} reporting the total computational time together with the percentage reduction for the multilevel schemes with different numbers of refinement levels relative to the single-level EARGD.  
	Figure~\ref{fig:multilevel_structure_gp} shows the structure of the  4-level EARGD scheme, illustrating the resulting adaptive cycles driven by the coarse condition~\eqref{eq:est_eta}. The multilevel cycles get progressively shallower, until at some point only regular gradient steps are taken.

	In Figure~\ref{fig:gp_groundstate1}, we present the optical lattice potential, the ground state computed by the 4-level EARGD, and the spatial evolution of the iterate errors $\|\phi_k-\phi_{\rm ref}\|$ for both schemes. The error is reduced first in the bulk, while the largest errors persist near the lattice sites and the domain boundary. The 4-level scheme reaches a comparable error after substantially fewer iterations than the single-level method.

		\begin{figure}[t]
			\hspace*{-13.5mm}\includegraphics{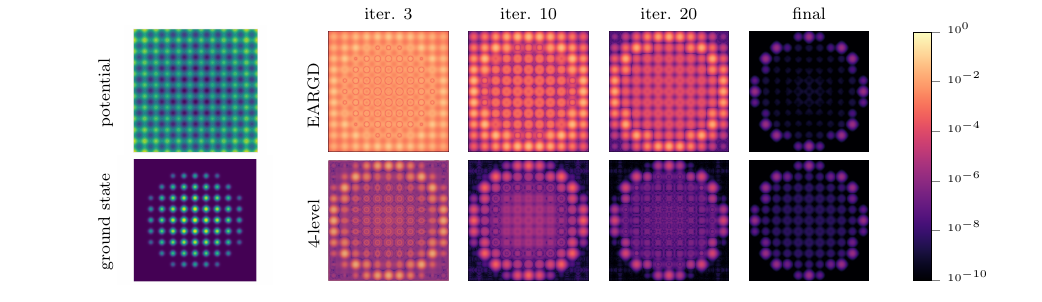}
			\caption{Gross--Pitaevskii model: (left) the harmonically confined
				optical lattice potential~$\vartheta$ and the ground state computed by the 4-level EARGD; (right) evolution of the iterate errors $\|\phi_k-\phi_{\rm ref}\|$
				for the single-level EARGD and 4-level EARGD  at the iterations $k=3, 10, 20$, and the final iterations ($k=130$ for EARGD and $k=56$ for the 4-level variant).}
			\label{fig:gp_groundstate1}
		\end{figure}
		
		\subsection{Binary continuous cuts problem}
		\label{sec:Binary-CC}
		
		Let $\Omega \subset \mathbb{R}^2$ be a bounded Lipschitz image domain and $g : \Omega \to \R$ a~grayscale image. The binary segmentation problem seeks a partition of $\Omega$ into foreground and background regions, encoded by a~binary label function $\phi\in \mathrm{BV}(\Omega;\{0,1\})$, i.e., $\phi=\chi_S$ a.e.
		for a set $S\subset\Omega$ of finite perimeter~\cite{AmbrosioFuscoPallara2000}. This problem is non-convex; a convex relaxation is obtained by replacing the binary 
		constraint with $\mathrm{BV}(\Omega;{[0,1]})$, yielding
		\begin{equation}\label{eq:cont_cuts}
			\begin{array}{l}
				\displaystyle{\min \;\calE^{\CC}(\phi) = \int_{\Omega} \rho(\xi)\,\phi(\xi)\,\drm\xi
					\;+\;
					\alpha \, \mathrm{TV}(\phi),}
				\\[4mm]
				\text{subject to } \; 0 \leq \phi(\xi) \leq 1, \quad \text{a.e. in } \Omega,
			\end{array}
		\end{equation}
		where $\alpha > 0$ is a regularization parameter, $\mathrm{TV}(\phi)$ denotes the total variation of $\phi$ on $\Omega$~\cite{AmbrosioFuscoPallara2000}, and 
		$\rho(\xi) = (c_f - g(\xi))^2 - (c_b - g(\xi))^2$ is the data term for characteristic foreground/background intensities $c_f,c_b$.
		Since $\rho$ enters $\calE^{\CC}$ linearly in $\phi$, the coarea formula implies that thresholding any global minimizer~$\phi^\ast$ of~\eqref{eq:cont_cuts}
		at almost every level 
		$\tau\in(0,1)$ yields a global binary minimizer 
		$\chi_{\{\phi^\ast>\tau\}}$ of the original binary problem~\cite{ChanEsedogluNikolova2006}. 
		
		\subsubsection{Finite difference discretization}
		\label{ssec:CC_discretization}
		We discretize the image domain on an~$m \times l$ pixel grid and identify both the image $g$ and segmentation variable $\phi$ with vectors in $\R^n$, where $n=ml$. For simplicity, we use the same notation as in the continuous setting. Adopting an isotropic discretization of the total variation, the minimization problem 
		\eqref{eq:cont_cuts} is replaced by its discrete and smoothed counterpart
		\begin{equation} \label{eq:cont_cuts_disc}
			\begin{array}{l}
				\displaystyle{\min \; 
					E^{\CC}_\varepsilon(\phi) = \rho^\top \phi + \alpha \sum_{i=1}^n \sqrt{(D_1 \phi)_i^2+(D_2 \phi)_i^2+\varepsilon^2}   }   \\[4mm]
				\text{subject to } \; 0 \leq \phi\leq 1,
			\end{array}
		\end{equation}
		where $\rho \in \R^n$ with $\rho_i = (c_f - g)_i^2 - (c_b - g)_i^2$,  $D  =\begin{bmatrix} D_{1} \\D_{2}\end{bmatrix} \in\R^{2n\times n}$ denotes the forward finite difference discretization of the gradient stacking horizontal and vertical differences, and $\varepsilon > 0$ is a small smoothing parameter.
		
		The box constraint in~\eqref{eq:cont_cuts_disc} is treated geometrically by passing to its interior $\calB=(0,1)^n.$ Each component $\phi_i$ is interpreted as the parameter of a Bernoulli distribution. Equipped with the Fisher--Rao metric, $\calB$ becomes the
		product Bernoulli manifold; see Table~\ref{tab:Bernoulli} and
		\mbox{\cite{Amari:2000,MuePZ23}}. This geometry keeps the iterates
		inside the open box without projection. Together with the objective, whose linear data term already favors binary labelings, the Fisher--Rao metric - singular as $\phi_i\to0,1$ - drives iterates toward near-binary labelings and reduces the need for post-processing thresholding.
		
		\begin{table}
			\begin{tabular}{lll}
				\toprule
				Bernoulli manifold &
				$\calB=(0,1)^n$
				\tabularnewline
				\midrule
				tangent space &
				$\displaystyle \Tan_\phi\, \calB=\R^n$ 
				\tabularnewline
				\midrule
				Fisher-Rao metric & $\langle u,v\rangle_\phi= \displaystyle{\onebb^\top \frac{u v}{\phi(\onebb-\phi)}}$, \; $u,v\in\Tan_\phi\,\calB$ 
				\tabularnewline
				\midrule
				Riemannian gradient&
				$\displaystyle
				\grad E^{\CC}_\varepsilon(\phi)
				=
				\diag\big(\phi(\onebb-\phi)\big) 
				\Big( \rho + \alpha D^\top ( I_2 \otimes \diag(\omega) ) D \phi \Big) 
				$
				\tabularnewline
				& with $\;
				\omega_i = \frac{1}{\sqrt{(D_1\phi)_i^2 + (D_2\phi)_i^2 + \varepsilon^2}}, \quad i = 1, \dots, n$
				\tabularnewline
				\midrule
				$e$-retraction &
				$\displaystyle 
				\Ret_\phi(v)
				=\frac{\phi\,\exp\!\Big(\frac{v}{\phi(\onebb-\phi)}\Big)}
				{\onebb-\phi+\phi\,\exp\!\Big(\frac{v}{\phi(\onebb-\phi)}\Big)}$
				\tabularnewline
				\midrule
				lifting map &
				$\displaystyle 
				\invRet_\phi(z)
				=\phi(\onebb-\phi)
				\log\!\Big(\frac{(\onebb-\phi)z}{\phi(\onebb-z)}\Big)$
				\tabularnewline
				\midrule
				differentiated lifting map &
				$\displaystyle 
				\Drm\invRet_\phi(z)[u]
				=\frac{\phi
					(\onebb-\phi)}{z(\onebb-z)}\,u$
				\tabularnewline
				\bottomrule
			\end{tabular}
			\caption{Geometric concepts for the product Bernoulli manifold $\calB$. Here, $I_2$ denotes the $2\times2$ identity matrix and $\otimes$ the Kronecker product.} 
			\label{tab:Bernoulli} 
		\end{table}
		
		\subsubsection{Vector transfer operators}
		\label{ssec:CC_vector_transport}
		Let $\calB_h=(0,1)^{n_h}$ and $\calB_H=(0,1)^{n_H}$ denote the fine and coarse
		Bernoulli manifolds, and let $I_h^H:\R^{n_h}\to\R^{n_H}$ and $I_H^h:\R^{n_H}\to\R^{n_h}$ be the corresponding restriction and prolongation operators between the coarse and fine grid discretizations. Using the logit and logistic functions 
		\begin{equation*}
			\logit(s)=\log\!\Big(\frac{s}{1-s}\Big), \qquad
			\sigma(t)=\frac{e^t}{1+e^t},
		\end{equation*}
		we define the point restriction and prolongation maps
		\[
		r(\phi)=\sigma\big(I^H_h\logit(\phi)\big), \qquad p(\psi)=\sigma\big(I_H^h\logit(\psi)\big),
		\]
		respectively. Here and in the following, all functions are understood componentwise. Thus, transfer between levels is performed in logit coordinates. Using the fine and coarse Fisher--Rao metric tensors
		\begin{align*}
			G_h(\phi)=\diag\!\Big(\frac{\onebb}{\phi(\onebb-\phi)}\Big),
			\qquad
			G_H(\psi)=\diag\!\Big(\frac{\onebb}{\psi(\onebb-\psi)}\Big),
		\end{align*}
		where $\onebb$ denotes the all-ones vector of appropriate dimension.
		Their differentials are given~by
		\begin{align} \label{eq:Dm-Dp-Bernoulli}
			\Drm r(\phi) = G_H(\psi)^{-1}\, I^H_h\, G_h(\phi), & \qquad \psi = r(\phi),\\
			\Drm p(\psi) = G_h(\phi)^{-1} I^h_H\, G_H(\psi), & \qquad \phi = p(\psi). \notag
		\end{align}
		Note that $G_h(\phi)=\Drm\logit(\phi)$, so the Fisher--Rao factors in \eqref{eq:Dm-Dp-Bernoulli} are exactly the Jacobians of the coordinate change to logit coordinates.
		
		Following the standard multilevel literature \cite{BriHM00}, we use injection $J^H_h$  and full-weighting $F^H_h$ as the grid restriction $I^H_h$, and bilinear interpolation $B^h_H$ as the grid prolongation $I^h_H$. These operators satisfy $F_h^H=\frac{1}{4}(B_H^h)^\top$, and $B^h_H$ is a right inverses of~$J^H_h$. For full-weighting and bilinear interpolation, the metric conjugation in~\eqref{eq:Dm-Dp-Bernoulli} is non-trivial, motivating the geometrically weighted operators
		\begin{align*}
			\mathrm{g}F^H_h := G_H^{-1}(\psi)\,F^H_h\,G_h(\phi), \qquad
			\mathrm{g}B^h_H := G_h^{-1}(\phi)\,B^h_H\,G_H(\psi).
		\end{align*}
		In contrast, injection $J^H_h$ is particularly simple. Since $J^H_h$ commutes with any componentwise map, including logit and logistic functions, we have $r(\phi) = \sigma (J^H_h \logit(\phi)) = J^H_h \phi$ and $G_H(\psi)^{-1}\,J^H_h\,G_h(\phi) = J^H_h$.
		
		These constructions give rise to the following vector transfer operators from Table~\ref{tab:vector_ops}:
		\begin{itemize}
			\item {Version~I}: 
			$P^{\phi}_{\psi} = \Drm r(\phi)^{\ast}
			= \bigl(I^H_h\bigr)^{\!\top}$ and 
			$R^{\psi}_{\phi} = \Drm r(\phi)
			= G_H(\psi)^{-1}\,I^H_h\,G_h(\phi)$;
			
			\item {Version~II}: 
			$P^{\phi}_{\psi} = \Drm r(\phi)^{-}
			= G_h(\phi)^{-1}\bigl(I^H_h\bigr)^{-}G_H(\psi)$ and 
			$R^{\psi}_{\phi} = \bigl(\Drm r(\phi)^{-}\bigr)^{\ast}
			= \bigl(I^H_h\bigr)^{-\top}$;
			
			\item {Version~III}: 
			$P^{\phi}_{\psi} = \Drm p(\psi)
			= G_h(\phi)^{-1}I^h_H\,G_H(\psi)$ and 
			$R^{\psi}_{\phi} = \Drm p(\psi)^{\ast}
			= \bigl(I^h_H\bigr)^{\top}$;
			
			\item {Version~IV}: 
			$P^{\phi}_{\psi} = I^h_H$ and 
			$R^{\psi}_{\phi} = I^H_h$;
			
			\item {Version~V}: 
			$P^{\phi}_{\psi} = I^h_H$ and 
			$R^{\psi}_{\phi} = G_H(\psi)^{-1} I^H_h G_h(\phi)$.
		\end{itemize}
		Unlike the other versions which are fully determined by the canonical grid transfer operators, Version~II retains additional algebraic freedom through the selection of the right inverse~$(I^H_h)^-$. Although any right inverse is admissible, the resulting vector transfer operators may vary in quality. For $I^H_h = J^H_h$, we choose the principled $B^h_H$. In the case $I^H_h = F^H_h$, we employ the Moore--Penrose inverse and its geometrically weighted form 
		\[
		(F^H_h)^+ := (F^H_h)^{\!\top}\bigl(F^H_h\,(F^H_h)^{\!\top}\bigr)^{-1}, \qquad
		\mathrm{g}(F^H_h)^+ := G_h^{-1}(\phi)\,(F^H_h)^+\,G_H(\psi).
		\]
		The resulting point restriction maps and vector transfer operators are collected in Table~\ref{tab:CC_vectort_ransports}.
		
		\begin{table}[h]
			\centering
			\label{tab:options}
			\renewcommand{\arraystretch}{1.4}
			\begin{tabular}{cc ccc ||c c c c}
				\hline
				{Option} & {Version} & $I^H_h$ & $(I^H_h)^-$ & $I^h_H$
				& & $r$ & $P_\psi^\phi$ & $R_\phi^\psi$  \\
				\hline
				\multirow{2}{*}{1}
				& II  & $J^H_h$ & $B^h_H$     & --- &      & \multirow{2}{*}{$J^H_h$}
				& \multirow{2}{*}{$\mathrm{g}B^h_H$}
				& \multirow{2}{*}{4$F^H_h$}\\
				& III & $J^H_h$ & ---         & $B^h_H$ &   & & & \\
				\hline
				\multirow{2}{*}{2}
				& I & $F^H_h$ & ---         & ---  &     & \multirow{2}{*}{$\mathrm{g}F^H_h$}
				& \multirow{2}{*}{$\frac{1}{4}B^h_H$} 
				& \multirow{2}{*}{$\mathrm{g}F^H_h$} \\
				& V & $F^H_h$ & ---          & $\frac{1}{4}B^h_H$ &    & & & \\
				
				\hline
				3   & II  & $F^H_h$ & $(F^H_h)^+$ & ---  &     & $\mathrm{g}F^H_h$ & $\mathrm{g}(F^H_h)^+$ & $\big((F^H_h)^+\big)^\top$ \\
				\hline
				4   & III   & $F^H_h$     & ---     & $B^h_H$ &   & $\mathrm{g}F^H_h$ & $\mathrm{g}B^h_H$ & $4\,F^H_h$  \\
				\hline
				5   & IV   & $F^H_h$     & ---     & $B^h_H$ &   & $\mathrm{g}F^H_h$ & $B^h_H$ & $F^H_h$  \\
			\end{tabular}
			\caption{The five vector transfer options. The left block specifies the construction choices, while the right block shows the resulting triples $(r,P_\psi^\phi,R_\phi^\psi)$.
				Multiple rows within the same option yield identical triples.}
			\label{tab:CC_vectort_ransports}
		\end{table}

		\subsubsection{Riemannian coarse model} 
		Consider the fine-level iterate $\phi_k\in\calB_h$ and its restriction \mbox{$\psi_k=r(\phi_k)\in\calB_H$}. 
		To construct the lifting map required in the coarse model, we employ the $e$-retraction defined in Table~\ref{tab:Bernoulli}.
		Expressed in terms of the inverse metric tensor and the logit function, the lifting map is then given by $\invRet_{\psi_k}^H(z) = G_H(\psi_k)^{-1}\big(\logit(z) - \logit(\psi_k)\big)$. Consequently, the metric tensor cancels out, and the geometric correction term simplifies to
		\begin{equation*}
			\big\langle w_k, \invRet^H_{\psi_k}(z) \big\rangle_{\psi_k} = w_k^\top \big( \logit(z) - \logit(\psi_k) \big).    
		\end{equation*}
		As a result, this term is affine in logit coordinates and, by construction, independent of the underlying metric. 
		The objective function of the Riemannian coarse model for the binary continuous cuts problem at $\phi_k$ is therefore given by
		\[
		\min_{z \in \calB_H} q_k^{\CC}(z) = E^{\CC}_H(z) - \big\langle \grad E^{\CC}_H(\psi_k) - R_{\phi_k}^{\psi_k} \big(\grad E^{\CC}_h(\phi_k)\big) , \invRet^H_{\psi_k}(z) \big\rangle_{\psi_k},
		\]
		where $R_{\phi_k}^{\psi_k}$ is the vector restriction operator, whose variants are given in Table~\ref{tab:CC_vectort_ransports}.
		
		\subsubsection{Numerical experiments}
		All numerical experiments for the binary continuous cuts problem were carried out in {\tt Python}. We test our scheme on the grayscale "two cows" image shown in Figure~\ref{fig:cow-grid} of size $960 \times 1280$ pixels. The objective is to segment the cows from the background using the given intensity patches. In the following, we present a systematic study of the vector transfer operators listed in Table~\ref{tab:CC_vectort_ransports} within the multilevel framework. We employ a $2$--level structure using a $240 \times 360$ coarsened version of the original image, with the fine and coarse level binary cuts objectives specified by $(\alpha_h, \varepsilon_h) = (0.1, 10^{-4}) $ and $(\alpha_H, \varepsilon_H) = (0.4, 10^{-3})$, respectively. For geometric coarse solves and fine-level updates, we employ Riemannian gradient descent (RGD), with the $e$--retraction, which coincides with mirror descent under the Fermi-Dirac entropy \cite{Raskutti2015,Raus2024}.  The coarse-level solver runs for $10$ iterations per call. The single-level variant uses the same solver. All step sizes are determined using Armijo backtracking. We compute a~reference energy $E_{\rm ref}$ using a~highly accurate optimization method. For this test model, we do not consider multilevel schemes with more than one coarse level, since the $2$-level setup exhibits better performance\footnote{An experiment demonstrating this is provided in the GitHub repository.}.
		
		
		\begin{figure}[t]
			\centering
			\hspace*{-4mm}\includegraphics[width=160mm]{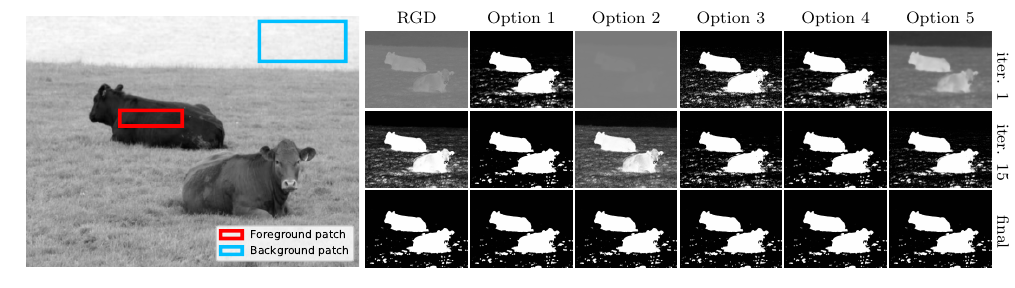}
			\caption{Continuous cuts segmentation on a 'two cow' image from the Microsoft Research Cambridge Object Recognition Image Database (\url{https://www.microsoft.com/en-us/research/project/image-understanding/}). (Left) Grayscale input with the foreground and background seed patches defining the data term. (Right) Evolution of the iterate $\phi_k$ at $k=1$, $15$ and the final iteration $k = 1000$ for the single-level RGD and the multilevel Options~1--5. Options~1 and ~4 dominate, segmenting well within a~few iterations, and Option~3 starts strongly but tapers off. All methods are reliable after a~few iterations, notably without thresholding.}
			\label{fig:cow-grid}
		\end{figure}
		
		Figures~\ref{fig:BCC_iter} and~\ref{fig:BCC_CPU} show the relative optimality gap and the stationarity measure for the tested methods as functions of the iteration count and CPU time, respectively. The geometric weighting of the vector prolongation operator proves to be crucial for performance. Without the metric conjugation,
		the prolongated corrections in Options~2 and~5 become excessively large near boundary pixels, where the Fisher--Rao metric attains high values, causing Armijo backtracking to select very small step sizes and effectively stalling the multilevel scheme. This issue is avoided by Options~1 and~4, for which the conjugation rescales the coarse correction direction before the fine-level line search. As a~result, these variants achieve the best multilevel performance and outpace the single-level variant. Among the metrically conjugated options, injection as point restriction further outperforms full-weighting. The latter averages in logit space and blurs the sharp edges of the fine-level iterate, whereas injection exactly reproduces the fine-grid configuration at the coarse-grid points while requiring less computational effort. 
		
		\captionsetup{width=.95\textwidth, skip=3pt}
		\begin{figure}[t]
			\includegraphics{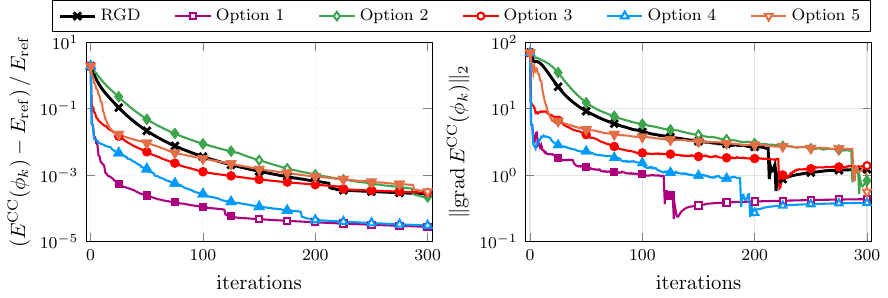}
			\caption{Binary continuous cuts: convergence histories of the relative energy error (left) and the stationarity measure (right) versus iteration count for different optimization schemes. Markers are drawn every $25$ iterations - filled ones indicate coarse correction steps. All methods exhibit sublinear convergence. The multilevel variants with geometrically weighted prolongations (Options~1 and~4) achieve significant acceleration over their single-level counterpart.} 
			\label{fig:BCC_iter}
		\end{figure}
		
		\captionsetup{width=.95\textwidth, skip=3pt}
		\begin{figure}[t]
			\includegraphics{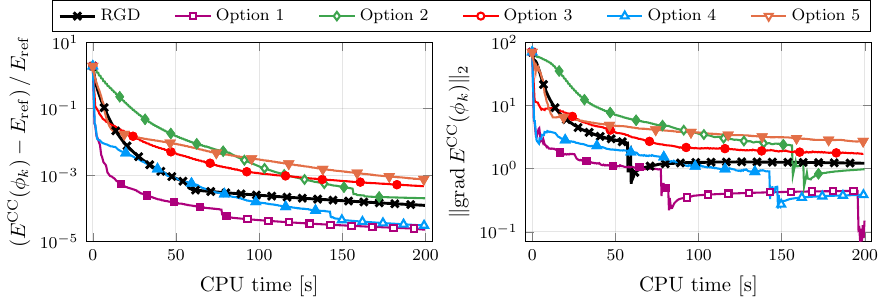}
			\caption{Binary continuous cuts: convergence histories of the relative energy error (left) and the stationarity measure (right) versus CPU time. Markers are drawn every $25$ iterations - filled ones indicate coarse correction steps. The computationally inexpensive vector transfer operators in Option~1 help it retain its acceleration advantage in CPU time.} 
			\label{fig:BCC_CPU}
		\end{figure}
		
		Although Option~3 employs a~geometrically weighted prolongation operator, its vector restriction is less effective. Relying on the Moore--Penrose inverse, it neither provides the smoothing effect of standard restriction operators nor admits a closed-form expression. Consequently, the coarse model inherits this lack of smoothness, and coarse corrections are triggered very frequently because the fine-gradient norm is maximally preserved, disturbing the coarse correction condition~\eqref{eq:est_eta} for essentially any choice of $\eta$; see Figure~\ref{fig:BCC_cc_bars}. Moreover, both vector restriction and prolongation require solving $(F^H_h (F^H_h)^\top)^{-1}$, making Option~3 significantly more expensive in CPU time than indicated by its iteration count, even with an~efficient FFT-based implementation; see Figures~\ref{fig:BCC_iter} and~\ref{fig:BCC_CPU}. Finally, we note that all options produce reliable segmentation after only a few iterations; cf. Figure~\ref{fig:cow-grid}. Thus, the convergence study is primarily included to provide a~systematic comparison of different optimization schemes. 
		
		\captionsetup{width=.95\textwidth, skip=3pt}
		\begin{figure}[t]
			\includegraphics{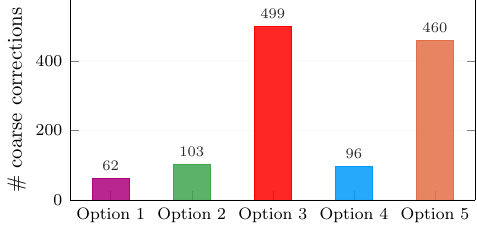}
			\caption{Binary continuous cuts: number of coarse corrections triggered during a~$1000$-iteration run. Options~1,~4 and~5 use a~common parameter choice $\eta = 0.6$ and $\mu = 0.5$. Option~2 with a~reduced threshold $\eta = 0.1$ still triggers relatively few coarse corrections. In contrast, Option~3 triggers them very frequently even with~$\eta = 0.85$, since its Moore--Penrose-based vector restriction nearly preserves the fine-gradient norm, causing condition~\eqref{eq:est_eta} to be satisfied for $\eta$ close to~1. These extreme parameter values indicate a~poorly balanced condition for the respective models.
			} 
			\label{fig:BCC_cc_bars}
		\end{figure}
		
		\subsection{Discussion}
		\label{subsec:discussion}
		
		We now compare the results for the considered three test examples,
		highlighting the key similarities and differences observed across the cases.
		
		\textit{Vector transfer choice.} For the Kohn--Sham and Gross--Pitaevskii problems, all tested operators yield similar results once the Galerkin condition is satisfied. For the continuous cuts model, metric conjugation in the vector prolongation is essential. Without it, prolongated corrections are oversized near binary pixels, forcing Armijo backtracking to take negligibly small steps and stalling convergence. This shows that the importance of metric-compatible transfer is governed by the non-uniformity of the metric tensor across the manifold. On the Bernoulli manifold~$\calB$, the Fisher--Rao metric varies by orders of magnitude between the interior and the near-binary boundary, making metric conjugation in the vector transfer operators essential, whereas 
		on the Stiefel and ellipsoid manifolds the metric is more uniformly 
		bounded and the impact is moderate.
		
		\textit{Levels trade-off.} More levels reduces iteration counts but increases per-step cost. The optimal number depends on the degree of freedom ratio between levels. For the Kohn--Sham problem, the 4-level scheme outperforms the 7-level scheme in CPU time. An analogous trade-off is observed for the Gross--Pitaevskii model, see Figure~\ref{fig:gp_time2}. For the binary continuous cuts problem, a~2-level scheme seems to be the sweet spot.
		
		\textit{Adaptive versus fixed cycling.} The condition~\eqref{eq:est_eta} lets the cycling pattern emerge adaptively rather than being imposed a~priori like a fixed V- or W-cycle, see Figures~\ref{fig:multilevel_structure} and~\ref{fig:multilevel_structure_gp}. Whenever the restriction reflects a genuine coarse-level contribution, the cost of setting up the coarse model is outweighed by the resulting speed-up, which can be verified numerically for all experiments.
		
		\textit{Coarse correction triggering.} The norm of the restricted fine gradient decreases monotonically and serves as a reliable stopping criterion for coarse corrections. Operators that maximally preserve the fine gradient norm (e.g., Option~3 in continuous cuts) trigger coarse corrections far too frequently, indicating poor balance in the correction condition rather than effectiveness.
		
		\textit{Convergence rates.} For the Kohn--Sham and Gross--Pitaevskii problems, RGD exhibits linear convergence, whereas for the continuous cuts model, it converges  sublinear due to the non-trivial kernel of the image gradient operator~$D$. Consequently, the objective remains far from a~strongly convex regime, even in the vicinity of the minimizer. This difference reflects the problem structure rather than the multilevel framework. Since consecutive coarse corrections are not permitted -- every coarse correction is preceded by a~fine-level gradient step (cf.~\eqref{eq:no_two_C}) -- the multilevel iterates are plausibly governed by the same asymptotic rate as single-level RGD; a~rigorous characterization of this inherited-rate phenomenon is left open.
		
		\textit{Discretization independence.} The framework accommodates plane-wave~(Kohn--Sham), finite element~(Gross--Pitaevskii), and grid-based finite difference~(continuous cuts) discretizations without modification, confirming that the Riemannian coarse model construction is agnostic to the spatial discretization strategy.
		
		\section{Conclusion}
		\label{sec:conclusion}
		
		We have presented a Riemannian multilevel optimization framework for constrained energy minimization problems on Riemannian manifolds. The approach is inspired by the optimization formulation of multilevel correction methods, but is formulated independently of a particular grid structure. It combines a Riemannian coarse model with a systematic family of metric-compatible vector transfer operators.
		
		A key feature of the proposed framework is that, under suitable compatibility conditions on the vector transfer operators, the coarse model can be made independent of the particular choice of Riemannian metric. In addition to the algorithmic developments, we established a~rigorous convergence theory for the proposed framework, proving global convergence under standard assumptions.
		
		Numerical experiments on the Kohn--Sham, Gross--Pitaevskii, and binary continuous cuts problems demonstrated substantial computational savings compared with single-level Riemannian optimization. The results also confirm that metric-compatible vector transfer is essential, in particular for manifolds with strongly non-uniform metrics, such as the Fisher--Rao metric on the Bernoulli manifold.
		
		These results indicate that the proposed multilevel framework is a promising and broadly applicable strategy for large-scale Riemannian optimization problems. Future work includes a~quantitative convergence rate analysis, developing adaptive coarsening strategies, and investigating its integration with second-order optimization methods.
		
		
		\printbibliography

@book{Trottenberg2001,
  author    = {Trottenberg, U. and Oosterlee, C. W. and Sch{\"u}ller, A.},
  title     = {Multigrid},
  publisher = {Academic Press},
  address   = {San Diego, CA},
  year      = {2001}
}

@article{RouxLeclercHild2014,
  author  = {Roux, S. and Leclerc, H. and Hild, F.},
  title   = {Efficient binary tomographic reconstruction},
  journal = {J. Math. Imaging Vis.},
  volume  = {49},
  pages   = {335--351},
  year    = {2014},
  doi     = {10.1007/s10851-013-0465-0}
}

@book{AmbrosioFuscoPallara2000,
  author    = {Ambrosio, L. and Fusco, N. and Pallara, D.},
  title     = {Functions of Bounded Variation and Free Discontinuity Problems},
  publisher = {Oxford University Press},
  year      = {2000}
}

@article{Raskutti2015,
    AUTHOR = {Raskutti, G. and Mukherjee, S.},
     TITLE = {The information geometry of mirror descent},
   JOURNAL = {IEEE Trans. Inform. Theory},
  FJOURNAL = {Institute of Electrical and Electronics Engineers.
              Transactions on Information Theory},
    VOLUME = {61},
      YEAR = {2015},
    NUMBER = {3},
     PAGES = {1451--1457},
    MRCLASS = {94A15},
  MRNUMBER = {3318003},
MRREVIEWER = {Ulrich\ Tamm},
       DOI = {10.1109/TIT.2015.2388583},
       URL = {https://doi.org/10.1109/TIT.2015.2388583},
}

@article{Raus2024,
    AUTHOR = {Raus, M. and Elshiaty, Y. and Petra, S.},
     TITLE = {Accelerated {B}regman divergence optimization with {SMART}: an
              information geometric point of view},
   JOURNAL = {J. Appl. Numer. Optim.},
  FJOURNAL = {Journal of Applied and Numerical Optimization},
    VOLUME = {6},
      YEAR = {2024},
    NUMBER = {1},
     PAGES = {1--40},
      MRCLASS = {65K05 (90C52)},
  MRNUMBER = {4876337},
  doi = {10.23952/jano.6.2024.1.01},  
}

@article{Brandt1977,
  author  = {Brandt, A.},
  title   = {Multi-level adaptive solutions to boundary-value problems},
  journal = {Math. Comput.},
  year    = {1977},
  volume  = {31},
  number  = {138},
  pages   = {333--390},
  doi     = {10.1090/S0025-5718-1977-0431719-X}
}

@article{ChanEsedogluNikolova2006,
  author  = {Chan, T.F. and Esedoglu, S. and Nikolova, M.},
  title   = {Algorithms for finding global minimizers of image segmentation and denoising models},
  journal = {SIAM J. Appl. Math.},
  volume  = {66},
  number  = {5},
  pages   = {1632--1648},
  year    = {2006},
  publisher = {SIAM},
  doi     = {10.1137/040615286}
}

@article{RiWi2012,
    author = {Ring, W. and Wirth, B.},
    title = {Optimization methods on {R}iemannian manifolds and their application to shape space},
    journal = {SIAM J. Optim.},
    volume = {22},
    number = {2},
    pages = {596-627},
    year = {2012},
    doi = {10.1137/11082885X},
    URL = {https://doi.org/10.1137/11082885X},
}

@article{WenGoldfarb2010,
  author  = {Wen, Z. and Goldfarb, D.},
  title   = {A line search multigrid method for large-scale nonlinear optimization},
  journal = {SIAM J. Optim.},
  volume  = {20},
  number  = {3},
  pages   = {1478--1503},
  year    = {2010},
  doi     = {10.1137/08071524X}
}

@article{Nash:2000,
    author    = {Nash, S.G.},
    title     = {A multigrid approach to discretized optimization problems},
    journal   = {Optim. Method. Softw.},
    volume    = {14},
    pages     = {99--119},
    year      = {2000},
    doi       = {10.1080/10556780008805795}
}

@phdthesis{wang2023posteriori,
  title={A posteriori error estimation for electronic structure calculations using ab initio methods and its application to reduce calculation costs},
  author={Wang, Y.},
  year={2023},
  school={Sorbonne Universit{\'e}},
  url = {https://theses.hal.science/tel-04524053v1},
}

@Book{AbsMS08,
	Author    = {Absil, P.-A. and Mahony, R. and Sepulchre, R.},
	Title     = {Optimization Algorithms on Matrix Manifolds},
	Publisher = {Princeton University Press},
	Address   = {Princeton, NJ},
	Year      = {2008}
}

@phdthesis{Hua13,
	author={Huang, W.},
	title={Optimization algorithms on {R}iemannian manifolds with applications},
	type = {{PhD} thesis},
	school = {Florida State University},
	year={2013},
	url={https://www.math.fsu.edu/~whuang2/pdf/Huang_W_Dissertation_2013.pdf}
}

@article{XuXXF21,
  author={Xu, F. and Xie, H. and Xie, M. and Yue, M.},
  title={A multigrid method for the ground state solution of {B}ose-{E}instein condensates based on {N}ewton iteration},
  journal={BIT Numer. Math.},
  volume={61},
  number={},
  pages={645-663},
  year={2021},
  doi={10.1007/s10543-020-00830-3}
}

@article{HenMP14,
    author = {Henning, P. and M{\aa}lqvist, A. and Peterseim, D.},
    title = {Two-level discretization techniques for ground state computations of {B}ose-{E}instein condensates},
    journal = {SIAM J. Numer. Anal.}, 
    volume = {52},
    number = {4},
    pages = {1525--1550},
    year = {2014},
    doi = {10.1137/130921520},
}

@article{SutV2021,
    author = {Sutti, M. and Vandereycken, B.},
    title = {Riemannian multigrid line search for low-rank problems},
    journal = {SIAM J. Sci. Comput.},
    volume = {43},
    number = {3},
    pages = {A1803-A1831},
    year = {2021},
    doi = {10.1137/20M1337430},
    URL = {https://doi.org/10.1137/20M1337430}
}

@article{MuePZ23,
    author = {M\"uller, S. and Petra, S. and Zisler, M.},
    title = {Multilevel geometric optimization for regularised constrained linear inverse problems},
    journal = {Pure Appl. Funct. Anal.},
    volume = {8},
    number = {3},
    pages = {855-880},
    year = {2023},
    doi = {},
    URL = {http://yokohamapublishers.jp/online2/oppafa/vol8/p855},
}

@article{HeidW25,
      title={Iterative energy reduction {G}alerkin methods and variational adaptivity}, 
      author={Heid, P. and Wihler, T.P.},
      year={2025},
      journal ={Preprint arXiv 2509.09600 [math.NA]},
      doi={10.48550/arXiv.2509.09600}, 
}

@article{DFTKjcon,
  author = {Herbst, M.F. and Levitt, A. and Canc{\`e}s, E.},
  doi = {10.21105/jcon.00069},
  journal = {Proc. JuliaCon Conf.},
  title = {{DFTK}: {A} {J}ulian approach for simulating electrons in solids},
  volume = {3},
  number = {26},
  pages = {69},
  year = {2021},
}

@Book{Boumal23,
    author    = {Boumal, N.},
    title     = {An Introduction to Optimization on Smooth Manifolds},
    publisher = {Cambridge University Press},
    address = {Cambridge},
    year      = {2023},
}

@article{WenY13,
	author = {Wen, Z. and Yin, W.},
	title = {A feasible method for optimization with orthogonality constraints},
	journal = {Math. Program.},
	volume = {142},
	number = {},
	pages = {397--434},
	year = {2013},
	doi = {10.1007/s10107-012-0584-1},
}

@article{ZhaH04,
	author = {Zhang, H. and  Hager, W.W.},
	title = {A nonmonotone line search technique and its application to unconstrained optimization},
	journal = {SIAM J. Optim.},
	volume = {14},
	number = {4},
	pages = {1043--1056},
	year = {2004},
	doi = {10.1137/S1052623403428208},
}

@article{AnsM25,
    title = {Adaptive gradient descent on {R}iemannian manifolds with nonnegative curvature},
    author = {Ansari-{\"O}nnestam, A. and Malitsky, Y.},
    journal = {Preprint arXiv 2504.16724 [math.OC]},
    year = {2025},
    doi = {10.48550/arXiv.2504.16724}
}

@article{HenY25,
    author = {Henning, P. and Yadav, M.},
    title = {Convergence of a {R}iemannian gradient method for the {G}ross--{P}itaevskii energy functional in a rotating frame},
    journal = {ESAIM Math. Model. Numer. Anal.},
    volume = {59},
    number = {},
    pages = {1145-1175},
    year = {2025},
    doi = {10.1051/m2an/2025018},
}

@article{Roo51,
  title = {New developments in molecular orbital theory},
  author = {Roothaan, C.C.J.},
  journal = {Rev. Mod. Phys.},
  volume = {23},
  number = {2},
  pages = {69--89},
  year = {1951},
  publisher = {American Physical Society},
  doi = {10.1103/RevModPhys.23.69},
  url = {https://link.aps.org/doi/10.1103/RevModPhys.23.69}
}

@article{CanKL21,
author = {Canc\`{e}s, E. and Kemlin, G. and Levitt, A.},
title = {Convergence analysis of direct minimization and self-consistent iterations},
journal = {SIAM J. Matrix Anal. Appl.},
fjournal = {SIAM Journal on Matrix Analysis and Applications}, 	
volume = {42},
number = {1},
pages = {243--274},
year = {2021},
doi = {10.1137/20M1332864},
URL = {https://doi.org/10.1137/20M1332864},
}

@article{AltPS22,
	author = {Altmann, R. and Peterseim, D. and Stykel, T.},
	title = {Energy-adaptive {R}iemannian optimization on the {S}tiefel manifold},
	DOI= "10.1051/m2an/2022036",
	url= "https://doi.org/10.1051/m2an/2022036",
	journal = {ESAIM: Math. Model. Numer. Anal.},
	year = 2022,
	volume = 56,
	number = 5,
	pages = "1629-1653",
}

@article{AltPS24,
	author = {Altmann, R. and Peterseim, D. and Stykel, T.},
	title = {{R}iemannian {N}ewton methods for energy minimization problems of {K}ohn--{S}ham type},
	DOI= "10.1007/s10915-024-02612-3",
	url= "https://doi.org/10.1007/s10915-024-02612-3",
	journal = {J. Sci. Comput.},
	year = {2024},
	volume = {101},
	pages = {article 6},
}

@article{PetPS25,
	author = {Peterseim, D. and P\"uschel, J. and Stykel, T.},
	title = {Energy-adaptive {R}iemannian conjugate gradient method for density functional theory},
	journal = {{arXiv:2503.16225 [math.NA]}},
	year = 2025,
	volume = {},
	number = {},
	pages = {},
    doi = {10.48550/arXiv.2503.16225}
}

@article{SchRNB09,
  author = {Schneider, R. and Rohwedder, T. and Neelov, A. and Blauert, J.},
  title = {Direct minimization for calculating invariant subspaces in density functional computations of the electronic structure}, 
  journal = {J. Comput. Math.},
  fjournal = {Journal of Computational Mathematics},
  year = {2009},
  volume = {27},
  number = {2-3},
  pages = {360--387},
  doi = {},
  url = {http://global-sci.org/intro/article_detail/jcm/8577.html}
}

@Book{Lee18,
	author    = {Lee, J.M.},
	title     = {Introduction to Riemannian Manifolds},
    edition   = {2nd},
	publisher = {Springer International Publishing AG},
    address   = {Cham},
    year      = {2018},
}

@article{SegK24,
  author = {S\'eguin, A. and Kressner, D.},
  title = {Hermite interpolation with retractions on manifolds}, 
  journal = {BIT Numer. Math.},
  year = {2024},
  volume = {64},
  number = {},
  pages = {},
  eid = {42},
  doi = {10.1007/s10543-024-01023-y},
}

@article{BaoD04,
	author = {Bao, W. and Du, Q.},
	title = {Computing the ground state solution of {B}ose--{E}instein condensates by a~normalized gradient flow},
	journal = {SIAM J. Sci. Comput.},
	volume = {25},
	year = {2004},
	number = {5},
	pages = {1674--1697},
	doi = {10.1137/S1064827503422956},
}

@article{HenP20,
  author = {Henning, P. and Peterseim, D.},
  title = {Sobolev gradient flow for the {G}ross--{P}itaevskii eigenvalue problem: global convergence and computational efficiency},
  journal = {SIAM J. Numer. Anal.},
  volume = {58},
  year = {2020},
  number = {3},
  pages = {1744--1772},
  doi = {doi: 10.1137/18M1230463},
}

@article{DanP17,
	AUTHOR = {Danaila, I. and Protas, B.},
	TITLE = {Computation of ground states of the {G}ross--{P}itaevskii functional via {R}iemannian optimization},
	JOURNAL = {SIAM J. Sci. Comput.},
	FJOURNAL = {SIAM Journal on Scientific Computing},
	VOLUME = {39},
	YEAR = {2017},
	NUMBER = {6},
	PAGES = {B1102--B1129},
	ISSN = {1064-8275},
	DOI = {10.1137/17M1121974},
}

@article{DuL22,
	author = {Du, C.-E. and Liu, C.-S.},
	title = {Newton--{N}oda iteration for computing the ground states of nonlinear {S}chr{\"o}dinger equations},
	journal = {SIAM J. Sci. Comput.},
	fjournal = {SIAM Journal on Scientific Computing},
	volume = {44},
	number = {4},
	pages = {A2370--A2385},
	year = {2022},
	doi = {10.1137/21M1435793},
}

@article{ChenLLZ24,
author = {Chen, Z. and Lu, J. and Lu, Y. and Zhang, X.},
title = {On the convergence of {S}obolev gradient flow for the {G}ross--{P}itaevskii eigenvalue problem},
journal = {SIAM J. Numer. Anal.},
volume = {62},
number = {2},
pages = {667-691},
year = {2024},
doi = {10.1137/23M1552553},
}

@article{WuWB17,
    author={Wu, X. and  Wen, Z. and Bao, W.},
   title={A regularized {N}ewton method for computing ground states of {B}ose--{E}instein condensates},
   journal={J. Sci. Comput.},
   volume={73},
    number={},
  pages = {303-329},
   year={2017}, 
   DOI={10.1007/s10915-017-0412-0},
}

@article{AntLT17,
author = {Antoine, X. and Levitt, A. and Tang, Q.},
title = {Efficient spectral computation of the stationary states of rotating {B}ose--{E}instein condensates by preconditioned nonlinear conjugate gradient methods},
journal = {J. Comput. Phys.},
volume = {343},
pages = {92-109},
year = {2017},
doi = {10.1016/j.jcp.2017.04.040},
}

@article{FenT25,
    title = {On preconditioned {R}iemannian gradient methods for minimizing the {G}ross--{P}itaevskii energy functional: algorithms, global convergence and optimal local convergence rate},
    author = {Feng, Z. and Tang, Q.},
    journal = {Preprint arXiv 2510.13516 [math.NA]},
    year = {2025},
    doi = {10.48550/arXiv.2510.13516},
}

@article{BorH08,
author = {Borz\'{\i}, A. and Hohenester, U.},
title = {Multigrid optimization schemes for solving {B}ose--{E}instein condensate control problems},
journal = {SIAM J. Sci. Comput.},
volume = {30},
number = {1},
pages = {441-462},
year = {2008},
doi = {10.1137/070686135}
}

@Article{ABBBFGHHKMMSTUWW25,
  author  = {Arndt, D. and Bangerth, W. and Bergbauer, W. and Blais, B. and Fehling, M. and Gassm\"{o}ller, R.
             and Heister, T. and Heltai, L. and Kronbichler, M. and Maier, M. and Munch, P. and Scheuerman, M.
             and Turcksin, B. and Uzunbajakau, S. and Wells, D. and Wichrowski, M.},
  title   = {{The deal.II library, Version~9.7}},
  journal = {J. Numer. Math.},
  year    = 2025,
  volume  = 33,
  number  = 4,
  pages   = {403--415},
  doi     = {10.1515/jnma-2025-0115}
}

@article{HerSY25,
	author = {Hermann, M. and Stykel, T. and Yadav, M.},
	title = {Qualitative and quantitative analysis of Riemannian optimization methods for ground states of rotating multicomponent Bose-Einstein condensates},
	journal = {Preprint arXiv 2512.05939 [math.NA]},
	year = {2025},
	doi = {10.48550/arXiv.2512.05939}
}

@article{ShuA23,
author = {Shustin, B. and Avron, H.},
title = {Riemannian optimization with a preconditioning scheme on the generalized {S}tiefel manifold},
journal = {J. Comput. Appl. Math.},
volume = {423},
eid = {114953},
year = {2023},
doi = {10.1016/j.cam.2022.114953}
}

@Book{BriHM00,
	Author    = {Briggs, W.L. and  Henson, V.E. and McCormick, S.F.},
	Title     = {A Multigrid Tutorial},
    Edition   = {2nd},
	Publisher = {SIAM},
	Address   = {Philadelphia, PA},
	Year      = {2000}
}

@article{CanCHM18,
    author={Canc{\`e}s, E. and Chakir, R. and He, L. and Maday, Y.},
    title={Two-grid methods for a class of nonlinear elliptic eigenvalue problems}, 
    journal={ IMA J. Numer. Anal.}, 
    year={2018},
    volume={38},
    number={},
    pages={605-645},
    doi={10.1093/imanum/drw053}
}

@article{CanDMSV16,
author = {Canc{\`e}s, E. and Dusson, G. and Maday, Y. and Stamm, B. and Vohral\'{\i}k, M.},
title = {A perturbation-method-based post-processing for the planewave discretization of {K}ohn–{S}ham models},
journal = {J. Comput. Phys.},
volume = {307},
pages = {446-459},
year = {2016},
doi = {10.1016/j.jcp.2015.12.012}
}

@article{GriH25,
      title={An additive two-level parallel variant of the {DMRG} algorithm with coarse-space correction}, 
      author={Grigori, L. and Hassan, M.},
      year={2025},
      journal ={Preprint arXiv: 2505.23429 [math.NA]},
      doi={10.48550/arXiv.2505.23429}, 
}

@article{ChenDGHZ14,
author = {Chen, H. and Dai, X. and Gong, X. and He, L. and Zhou, A.},
title = {Adaptive finite element approximations for {K}ohn--{S}ham models},
journal = {Multiscale Model. Simul.},
volume = {12},
number = {4},
pages = {1828-1869},
year = {2014},
doi = {10.1137/130916096}
}

@article{XieXYZ23,
author = {Xie, H. and Xie, M. and Yin, X. and Zhao, G.},
title = {An efficient adaptive mesh redistribution method for nonlinear eigenvalue problems in {B}ose--{E}instein condensates},
journal = {J. Sci. Comput.},
volume = {94},
number = {},
pages = {article 37},
year = {2023},
doi = {10.1007/s10915-022-02093-2},
}

@book{Amari:2000,
	author = {Amari, S.-I. and Nagaoka, H.},
	publisher = {Americam Mathematical Society},
    address ={Providence},
	title = {{Methods of Information Geometry}},
	year = {2000}
 }

@article{ElshiatyP2026,
author = {Elshiaty, Y. and Petra, S.},
title = {Multilevel {B}regman proximal gradient descent},
journal = {SIAM J. Imaging Sci.},
volume = {19},
number = {2},
pages = {913-942},
year = {2026},
doi = {10.1137/25M1775725},
}
		
	\end{document}